\documentclass[11pt]{amsart}
\usepackage{amssymb, amsmath,graphicx}
\def\ov{\overline}
\def\dint{\int\!\!\int}

\def\inte#1{
\displaystyle\mathop{#1\kern0pt}^\circ }

\let\al=\alpha

\let\e=\varepsilon

\let\f=\frac

\let\p=\psi

\def\ga{\gamma}
\def\n{{\mathbf n}}
\def\ve{\varepsilon}
\def\meas{\hbox{meas}\,}

\def\virgp{\raise 2pt\hbox{,}}
\def\cdotpv{\raise 2pt\hbox{;}}

\def\eqdefa{\buildrel\hbox{\footnotesize def}\over =}

\def\C{\mathop{\bf C\kern 0pt}\nolimits}
\def\DD{\mathop{\bf D\kern 0pt}\nolimits}
\def\K{\mathop{\bf K\kern 0pt}\nolimits}
\def\N{\mathop{\bf N\kern 0pt}\nolimits}
\def\Q{\mathop{\bf Q\kern 0pt}\nolimits}
\def\R{\mathop{\mathbb R\kern 0pt}\nolimits}
\def\SS{\mathop{\bf S\kern 0pt}\nolimits}
\def\ZZ{\mathop{\bf Z\kern 0pt}\nolimits}
\def\TT{\mathop{\bf T\kern 0pt}\nolimits}

\newcommand{\la}{\lambda}

\def\nn{\nonumber}

\def\p{\partial}

\newcommand{\beq}{\begin{equation}}
\newcommand{\eeq}{\end{equation}}
\newcommand{\ben}{\begin{eqnarray}}
\newcommand{\een}{\end{eqnarray}}
\newcommand{\beno}{\begin{eqnarray*}}
\newcommand{\eeno}{\end{eqnarray*}}

\newtheorem{thm}{Theorem}[section]

\newtheorem{col}{Corollary}[section]
\newtheorem{prop}{Proposition}[section]

\begin{document}

\title{Singularity and existence to a wave system of nematic liquid crystals}
\author[G. Chen]{Geng Chen}%
\address[G. Chen]
{ Department of Mathematics\\
 The Pennsylvania State
University\\
 University Park, PA 16802, USA} \email{chen@math.psu.edu}
\author[Y. Zheng]{Yuxi Zheng}%
\address [Y. Zheng]
{ Department of Mathematics\\
 The Pennsylvania State
University\\
 University Park, PA 16802, USA}\email{yuz2@psu.edu}
\date{\today}
\maketitle

\begin{abstract}
In this paper, we prove the global existence and singularity formation for a wave system
from modelling nematic liquid crystals in one space dimension. 
In our model, although the viscous damping term is included,  
the solution with smooth initial data still has gradient blowup in general,
even when the initial energy is arbitrarily small. 
\end{abstract}
\bigskip

\bigskip

\setcounter{equation}{0}
\section{Introduction}
The mean orientation of the long molecules in a nematic liquid crystal is
described by a director field of unit vectors, ${\mathbf
n}\in{\mathbb S}^2$, and
the propagation of the orientation waves in the director field could be 
modelled by below Euler-Largrangian equations derived from the least action principle \cite{[37],AH}, 
\beq
 {\mathbf
n}_{tt} + \mu {\mathbf
n}_t + \frac{\delta W(\n,\nabla \n)}{\delta \mathbf n} =\lambda {\mathbf
n}, \qquad {\mathbf n}\cdot
{\mathbf n} = 1, \label{1.2} 
\eeq 
where the well-known Oseen-Franck potential
energy density $W$ is given by 
\beq 
 \textstyle
W\left({\mathbf n},\nabla{\mathbf
n}\right) = \frac12\alpha(\nabla\cdot{\mathbf n})^2 +
\frac12\beta\left({\mathbf n}\cdot\nabla\times{\mathbf n}\right)^2
+\frac12\gamma\left|{\mathbf n}\times(\nabla\times{\mathbf
n})\right|^2. \label{1.1a} 
\eeq 
The positive constants $\alpha$,
$\beta$, and $\gamma$ are elastic constants of the liquid crystal,
corresponding to splay, twist, and bend, respectively. 
The viscous coefficient $\mu$ is a non-negative constant and the Lagrangian multiplier $\lambda$ is  determined 
by the constraint $\mathbf n \cdot \mathbf n=1$.

There are many studies
on the constrained elliptic system of equations for ${\mathbf n}$
derived through variational principles from the Oseen-Franck potential \eqref{1.1a}, 
and on the parabolic flow associated with it, see
\cite{[3],[7],[10],[19],[26],[49]}.

The global weak existence and singularity formation for the Cauchy problem 
of the extreme case of (\ref{1.2}) 
with $\mu=0$ in one space dimension (1-d)
has also been extensively studied \cite{BZ,ghz, ZZ03,  ZZ05a, 
ZZ10, ZZ11, CZZ12}. This hyperbolic system describes the model with viscous effects neglected, 
for which 
an example with smooth initial data and singularity formation (gradient blowup) in finite 
time has been provided in \cite{ghz}. The lack of regularity makes us only be able to consider the existence of weak solutions
instead of classical solutions. The existing global existence results in \cite{BZ, ZZ10, ZZ11, CZZ12} proved by the method of 
energy-dependent coordinates will be introduced in Section \ref{section 2}.

However, the well-posedness and regularity of 
the solution for the complete system (\ref{1.2}) with $\mu>0$ are still wide open.

In this paper, we consider the 1-d system of \eqref{1.2} with $\alpha=\beta$ and $\mu\geq0$, 
where the derivation will be delayed to Section \ref{section 2}:
\beq \left\{
\begin{array}{ll}
\p_{tt}n_1+ \mu \p_{t} n_1 -\p_x(c^2(n_1)\p_xn_1)=\bigl(-|\n_t|^2+(2c^2-\ga)|\n_x|^2\bigr)n_1, \\
\p_{tt}n_2+ \mu \p_{t} n_2 - \p_x(c^2(n_1)\p_xn_2)=\bigl(-|\n_t|^2+(2c^2-\al)|\n_x|^2\bigr)n_2,\\
\p_{tt}n_3+ \mu \p_{t} n_3 - \p_x(c^2(n_1)\p_xn_3)=\bigl(-|\n_t|^2+(2c^2-\al)|\n_x|^2\bigr)n_3,
\end{array}\right. \label{1.1} \eeq
where $\mu\geq 0$,
\[
c^2(n_1)\eqdefa \al+(\ga-\al)n_1^2
\] 
and 
\beq\label{n1}
\n\equiv \n(t,x)=(n_1(t,x),n_2(t,x),n_3(t,x))\quad \mbox{and}\quad |\n|=1,\quad\mbox{where} \quad (t,x)\in\mathbb R^+\times\mathbb R.
\eeq
In this paper, we first provide an example with smooth ($C^1$) initial data and singularity formation in finite time.
Then we prove the global weak existence of the Cauchy problem for \eqref{1.1} with initial data
\beq\label{ID} 
n_i(0,x)={n_i}_0\in H^1,\quad
(n_i)_t(0,x)={n_i}_1\in L^2, \quad i=1\sim3.
\eeq

Although when $\mu=0$ the singularity formation and global existence of Cauchy problem 
for \eqref{1.1} has been systematically studied in \cite{ghz} and \cite{CZZ12} respectively,
the extension of these results to the case with $\mu>0$ in this paper is 
important and non-trival. First, the model with 
$\mu=0$ only describes the extreme case in physics which is much less 
happening than the one described by the model with $\mu>0$. Secondly, the comparison 
between our results and the results when $\mu=0$ indicates that the regularity 
of the solutions is basically \emph{not} impacted by adding or removing the viscous term $\mu\n_t$ even when 
$\mu$ is large, which is to some extend unexpected. Furthermore, now most 
existing existence and regularity results for the models with $\mu=0$ can be expected to extend to
the models with $\mu>0$. 

We first consider the regularity of solutions of (\ref{1.1}). 
A nature question is whether the appearance of the viscous term 
$\mu \n_t$ with $\mu>0$ can prevent the 
singularity formation (gradient blowup), which is known existing when $\mu=0$, or not. However, the singularity 
formation result in this paper gives a negative answer to this question: 
For any $\mu\geq 0$, the solution generally includes gradient blowup.  
In fact, for both the singularity formation example in \cite{ghz} ($\mu=0$) and the one in this paper ($\mu>0$),
the initial energy can be arbitrarily small.
As a consequence of this surprising result, we need to consider the weak 
solution instead of the classical solution for system \eqref{1.1}.

For simplicity, in order to get a gradient blowup example, 
we only consider the solution of \eqref{1.1} with structure
$\n=(\cos u(t,x), \sin u(t,x), 0)$ (planar deformation). Clearly $n_3$ is always $0$ 
in the smooth solution if it is vanishing initially. It is easy to get that the system (\ref{1.1}) is equivalent to
\beq\label{uc}
u_{tt}+\mu u_t - c(u){\left(\, c(u) u_x \right)}_x=0, 
\eeq 
when $\sin u\neq 0$, where we still use $c$ to denote the wave speed and
\beq
c^2(u)=\gamma \cos^2 u +\alpha \sin^2 u.
\eeq
It is easy to see that there exist positive constants $C_L$, $C_U$ and $C_D$, such that
\beq\label{sing_bounds}
C_L<c(u)<C_U,\quad |c'(u)|<C_{D}.
\eeq
{\begin{thm}\label{sing}
We consider the Cauchy problem of \eqref{uc} with initial data satisfying
\begin{eqnarray}
u(0,x)   &=&u_0+ \e\, \phi(\frac{x}{\e}) + \e^2\, \eta(\e^{\frac{2}{3}}x),\nonumber\\
u_t(0,x)&=&\left( - c(u(0,x))+\e \right)\, u_x(0,x),\label{id_sing0}
\end{eqnarray}
where $u_0$ is a constant satisfying
\[
c'(u_0)>0,
\]
and two functions $\phi(a)$ and $\eta(a)$ are in $C^1(\mathbb R)$ and satisfy conditions
\beq\label{id_sing1}
\phi(a),\, \eta(a)=0\quad \text{when}\quad a\not\in(-1,1);\quad \phi'(a)<0\quad \text{when}\quad a\in[0,1),
\eeq
and
\beq\label{id_sing2}
-  \phi'(0)>\frac{8\,\mu\, C_U}{ c'(u_0)\, C_L}\,.
\eeq
For any $\mu\geq 0$, we can choose $\e>0$ sufficiently small, 
such that the $C^1$ solution $u(t,x)$ forms singularity in finite time.
\end{thm}
}
Note, in the solutions we constructed, $\sin u\neq 0$ before blowup. Hence, for system \eqref{1.1}, for any $\mu\geq 0$, we can find examples with $C^1$ smooth initial data and
singularity formation in finite time. From Theorem \ref{1.1thm}, we know that the gradient blowup we found is \emph{not} a discontinuity (shock wave).

Then we consider the global existence of weak solution 
for the initial value problem of \eqref{1.1}$\sim$\eqref{ID}. 
Here, the unit vector ${\mathbf n}(t, x)$ with
$(t,x)\in\mathbb R^+\times\mathbb R$
is a  {\em weak solution} to the Cauchy problem \eqref{1.1}$\sim$\eqref{ID} if it satisfies that
{\em
\begin{itemize}
	\item[1.] The equations \eqref{1.1} hold in distributional sense for test functions $\phi\in
		C^1_c(\mathbb R\times\mathbb R)$. The initial data satisfy (\ref{ID})
		in the pointwise sense, and their temporal derivatives hold in  $L^p_{\rm loc}\,$ for $p\in [1,2)\,$.
	\item[2.] $n_i (t,x)$ is locally H\"older continuous with exponent $1/2$, for $i=1\sim3$.   
		 The map $t\mapsto (n_1, n_2, n_3)(t,\cdot)$ is continuously differentiable 
		with values in $L^p_{\rm loc}$, for all $1\leq p<2$. And it is Lipschitz continuous
		under the $L^2$ norm, i.e. \beq\label{1.lip}
		\big\|{n_i}(t,\cdot)-{n_i}(s,\cdot)\big\|_{L^2}
		\leq L\,|t-s|, \quad i=1\sim3,
		\eeq 
		for all $t,s\in\mathbb R^+$.
\end{itemize}
}

The main well-posedness results in this paper are
\begin{thm} \label{1.1thm}
The initial value problem \eqref{1.1}$\sim$\eqref{ID} has a global weak solution $\n(t,x)$ for all $(t, x)\in \mathbb R^+\times {\mathbb R}$, where $\n(t,x)$ satisfies that the energy 
\beq\label{1.5E} 
\mathcal{E}(t) 
\eqdefa  \frac{1}{2}\int
\Big[|\n_t|^2+c^2(n_1)|\n_x|^2\Big] \,dx \eeq
is less than or equal to $\mathcal{E}(0)$ for any $t>0$.
\end{thm}
\begin{thm}\label{cd}
For the Cauchy problem \eqref{1.1}$\sim$\eqref{n1},
let a sequence of initial data satisfy 
\beno
\big\|({n_i}_0^k)_x- ({n_i}_0)_x\big\|_{L^2}  \to 0,\qquad
 \big\|{n_i}_1^k- {n_i}_1\big\|_{L^2}  \to
0\,,\quad i=1\sim3,
\eeno 
and ${\mathbf n}^k_0 \to {\mathbf n}_0$ uniformly on compact sets,
as $k\to \infty$. Then ${\mathbf n}^k\to {\mathbf n}$ uniformly on bounded subsets of
the $(t,x)$-plane with $t>0$.
\end{thm}

In order to prove these theorems, by the method
of energy-dependent coordinates,  first used in
papers \cite{BZ} and related Camassa-Holm equation, we dilate
the singularity then find the energy dissipated solution. In fact, 
the energy is dissipated when $\mu<0$. When $\mu=0$, the energy is
conserved in some sense \cite{CZZ12}. 

Similar existence and continuous dependence results apply to system \eqref{uc}, 
which is a special example of \eqref{1.1}.

The paper is organized as follows.  In
Section 2, we derive the system \eqref{1.1} and energy equation, then introduce the 
existing results for the extreme case $\mu=0$. 
In Section 3, we prove the singularity formation result: Theorem \ref{sing}. 
In Section 4, we construct a set of semi-linear equations in the energy-dependent variables by studying 
the smooth solution of \eqref{1.1}. Then, in Section 5, we prove the existence of the solutions for the semi-linear system. 
In Sections 6$\sim$8, we transform the solutions back to the original system and prove Theorems \ref{1.1thm}$\sim$\ref{cd}.
\section{The derivation of system}\label{section 2}
\setcounter{equation}{0}
In this section, we derive (\ref{1.1}) from (\ref{1.2}).
In one space dimension, \eqref{1.2} is
\beq
\p_{tt}n_i + \mu\p_{t}n_i+\p_{n_i}W(\n,\p_x\n)-\p_x\bigl[\p_{\p_xn_i}W(\n,\p_x\n)\bigr]=\la
n_i,\quad \mbox{for}\quad i=1,2,3. \label{1.4} 
\eeq 
Using
$|\n|=1,$ we get by multiplying $n_i$  to \eqref{1.4} and summing up
$i$ from $1$ to $3$ \beq
\la=\sum_{i=1}^3\Bigl\{-|\p_tn_i|^2+n_i\p_{n_i}W(\n,\p_x\n)-n_i\p_x\bigl[\p_{\p_xn_i}W(\n,\p_x\n)\bigr]\Bigr\}.
\label{1.5} \eeq 
It is easy to calculate that \eqref{1.4} has the
following energy equations: \beq \label{1.4a}
\p_t\bigl[\f12|\n_t|^2+W(\n,\p_x\n)\bigr]-\p_x\bigl[\sum_{i=1}^3\p_tn_i\p_{\p_xn_i}W(\n,\p_x\n)\bigr]=-{\mu} |\n_t|^2\leq0.
\eeq
In 1-d case, the Oseen-Franck potential energy density \eqref{1.1a} is
\beq\label{1.6}
W(\n,\p_x\n)=\f{\al}2(\p_xn_1)^2+\f{\beta}2\bigl[(\p_xn_2)^2+(\p_xn_3)^2\bigr]+\f12(\ga-\beta)n_1^2|\p_x\n|^2,
\eeq from which and \eqref{1.5}, we infer \beq\label{1.7}
\la=-|\n_t|^2+\bigl(\beta+2(\ga-\beta)n_1^2\bigr)|\p_x\n|^2+(\beta-\al)n_1\p_x^2n_1.
\eeq Then by \eqref{1.4}, \eqref{1.6} and \eqref{1.7}, we
have
\beq\label{1.8}
\left\{
\begin{split}
&\p_{tt}n_1+ \mu\p_{t}n_1-\p_x\bigl[c_1^2(n_1)\p_xn_1\bigr]=
\bigl\{-|\n_t|^2+(2c_2^2-\ga)|\n_x|^2+2(\al-\beta)(\p_xn_1)^2\bigr\}n_1\\
&\p_{tt}n_2+\mu\p_{t}n_2-\p_x\bigl[c_2^2(n_1)\p_xn_2\bigr]=
\bigl\{-|\n_t|^2+(2c_2^2-\beta)|\n_x|^2+(\beta-\al)n_1\p_{xx}n_1\bigr\}n_2\\
&\p_{tt}n_3+\mu\p_{t}n_3-\p_x\bigl[c_2^2(n_1)\p_xn_3\bigr]=
\bigl\{-|\n_t|^2+(2c_2^2-\beta)|\n_x|^2+(\beta-\al)n_1\p_{xx}n_1\bigr\}n_3
\end{split}
\right.
\eeq 
with 
\beno c_1^2(n_1)\eqdefa
\al+(\ga-\al)n_1^2\quad\mbox{and}\quad
c_2^2(n_1)\eqdefa\beta+(\ga-\beta)n_1^2. \eeno

In particular, taking $\al=\beta$ in (\ref{1.8}), 
we get \eqref{1.1} and
\beq
c^2(n_1)=c^2_1(n_1)=c^2_2(n_1)=\al+(\ga-\al)n_1^2,
\eeq
and it is easy to check that there are positive numbers $C_L<C_U$ such that
\beq\label{1.0c} 0< C_L < c(n_1) < C_U < \infty; \quad
|c'(n_1)| < C_N,\quad \text{for\ any}\quad |n_1|\leq1.  \eeq
The energy equation \eqref{1.4a} is 
\beq \label{1.4a2}
\f12\p_t\bigl[|\n_t|^2+c^2(n_1)|\n_x|^2]-\p_x\bigl[c^2(n_1)\n_t\cdot\n_x]=-\mu |\n_t|^2\leq0,
\eeq
where the energy density $W$ in \eqref{1.6} is
\beq
W({\mathbf n}, \partial_x {\mathbf n})=\frac{1}{2}c^2(n_1)|\partial_x {\mathbf n}|^2.
\eeq 

Finally, we introduce the existing existence results for weak solutions of 1-d
systems of \eqref{1.2} with $\mu=0$, proved by the method
of energy-dependent coordinates. This method was first applied to equation \eqref{uc} with $\mu=0$ for an existence proof by
\cite{BZ}. When $\n$ is arbitrary in $\mathbb S^2$ and $\mu=0$, under the a priori assumption that $\n(x,t)$ is uniformly away from $(1,0,0)$, the existences of 1-d solutions for (\ref{1.1}) and (\ref{1.8}) with $\beta<\alpha$ have been provided by \cite{ZZ10} and \cite{ZZ11}, respectively. In a recent paper  \cite{CZZ12}, the existence of 1-d solution for system \eqref{1.1} and $\mu=0$ has been established, without any a priori assumption.
\section{Singularity formation: Proof of Theorem \ref{sing}}
\setcounter{equation}{0}
We introduce Riemann variables $R$ and $S$ for equation \eqref{uc}:
\beq
R\eqdefa u_t + c(u)\, u_x,\quad S\eqdefa u_t - c(u)\, u_x.
\eeq
By \eqref{uc}, the $C^1$ solution satisfies
\ben
R_t-c R_x &=&\frac{c'(u)}{4c(u)} (R^2-S^2)-\frac{\mu}{2}(R+S),\label{R_sing}\\
S_t+c S_x &=&\frac{c'(u)}{4c(u)} (S^2-R^2)-\frac{\mu}{2}(R+S),\label{S_sing}
\een
where the left hand sides
are two directional derivatives along characteristics. 

Before we give the proof, we first explain the basic idea. 
The singularity is essentially caused
by the quadratic increase of the Riemann variables. In order to find it out, we need to get a
decoupled Riccati type inequality on some characteristics. 
First, we show that the energy is in $O(\e)$ for our example. 
Using this estimate, we can make sure that $c'(u)>\frac12 c'(u_0)$ before $t=\sigma/\e$ for some positive constant $\sigma$,
when $\e>0$ is sufficiently small. This means that the coefficients of the quadratic terms in  (\ref{R_sing})(\ref{S_sing}) 
have fixed sign. Then the key step of this proof is to prove the domain: $0\geq R>-M_2 \e^{\frac16},\ S>0$ is invariant, 
when $(t,x)$ is taken on some characteristic trapezoid. Collecting all these informations, we can find the decoupled Riccati type inequality then prove the blowup.

We will prove the Theorem \ref{sing} in several steps.\bigskip

\paragraph{\bf{1}}
It is easy to get from the initial condition that
\ben
u_x(0,x)&=&\phi'(\frac{x}{\e}) + \e^{\frac{8}{3}}\, \eta'(\e^{\frac{2}{3}}x),\\
R(0,x)&=&\e\, u_x(0,x),\label{R_0_sing}\\
S(0,x)&=&\left( - 2\,c(u(0,x))+\e \right)\, u_x(0,x).\label{S_0_sing}
\een
Choosing $\e<C_L$ and $0<\e\ll 1$, we have, when $x\in[0, \e^{-\frac{2}{3}})$
\beq\label{RSIC_sing}
u_x(0,x)<0, \quad R(0,x)<0, \quad S(0,x)>0,
\eeq
because $\phi'(a)$ and $\eta'(a)$ are both negative when $a\in[0,1)$.  

\bigskip
\paragraph{\bf{2}}
By \eqref{R_sing} and \eqref{S_sing}, we have the energy equation
\beq\label{E_ineq}
(R^2+S^2)_t +(c(S^2-R^2))_x=-\mu(S+R)^2\leq 0.
\eeq
This equation agrees with the energy equations \eqref{1.4a} and \eqref{1.4a2}.
We define the energy function $E$ as
\beq
E(t)\equiv E(u(\cdot, t))=\int_{-\infty}^{\infty}u^2_t(x, t)+c^2(u(x, t))u^2_x(x, t)\,dx
=\frac12\int_{-\infty}^{\infty} R^2(t, x)+S^2(t, x)\, dx.
\eeq
Here we use $E$ to denote energy for equation (\ref{uc}) instead of $\mathcal E$ for \eqref{1.1} to avoid confusion.
Integrate \eqref{E_ineq} with respect to $x$, then we have
\beno
E(t)&\leq& E(0)\\
	& =  & \frac12\int_{-\infty}^{\infty} R^2(0,x)+S^2(0,x)\, dx\\
	& =  & \frac12\int_{-\infty}^{\infty} \big[\,\left( - 2\,c(u(0,x))+\e \right)^2+\e^2\,\big]\, 
	\big(\,\phi'(\frac{x}{\e}) + \e^{\frac{8}{3}}\, \eta'(\e^{\frac{2}{3}}x)\,\big)^2 \, dx\\
	&\leq& M\e,
\eeno
for some positive constant $M$, where we use $0<\e\ll1$.

\begin{figure}[htb]
\centering
\includegraphics[width=7cm,height=4cm]{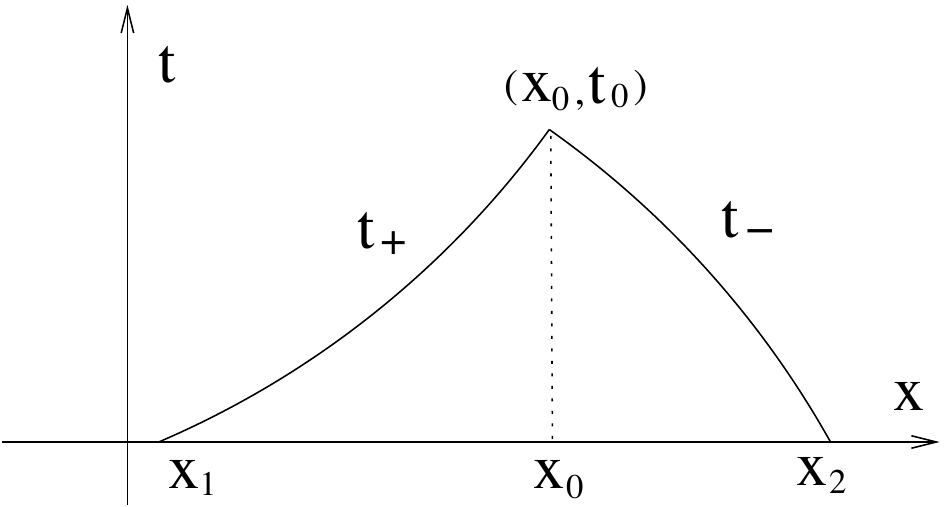}
\caption{Characteristic triangle}
\label{f0}
\end{figure}

\bigskip
\paragraph{\bf{3}}
We consider any characteristic triangle in Figure \ref{f0} 
with the characteristic boundaries $t_{\pm}$ denoted by 
\[
\frac{dt_{\pm}(x)}{dx}=\pm \frac{1}{c(u)}.
\]
Integrating \eqref{E_ineq} on this characteristic triangle, 
and by the divergence theorem, we have
\beq\label{dependence}
\int_{x_1}^{x_0}R^2(t_{+}(x),x)\,dx+\int_{x_0}^{x_2}S^2(t_{-}(x),x)\,dx\leq
\frac12 \int_{x_1}^{x_2} \left[R^2(0,x)+S^2(0,x)\right]\,dx.
\eeq

\bigskip
\paragraph{\bf{4}}
Then we control the sign of $c'(u)$ by choosing $\e$ small enough. Since equation \eqref{uc} has finite propagation speed, 
clearly seen from \eqref{dependence}, we get 
that $u=u_0$ in the region $x<-\e^{-\frac2 3}-C_U t$ and $x>\e^{-\frac2 3}+C_U t$.
\ben
|u(t,x)-u_0|&=&|\int_{-\infty}^x u_x(t,x)\, dx|\nn\\
&\leq& \int_{-\e^{-\frac2 3}-C_U t}^{\e^{-\frac2 3}+C_U t}
|u_x(t,x)| \, dx\nn\\
&\leq& \|u_x \|_{L^2} \, \sqrt{2\e^{-\frac{2}{3}}+2\, C_U t}\nn\\
&\leq& \frac{\sqrt{M}}{C_L}\sqrt{2\e^{\frac{1}{3}}+2\, C_U\, t\,\e},
\een
where we use the bound on $E(t)$ in part 2.
Hence, we can find a small positive number $\sigma$ independent of $\e$ then another small positive number $\e_0$, such that,
\beq\label{cd_sign}
c'(u(x,t))>\frac{c'(u_0)}2>0,
\eeq
if 
\[
0\leq t <\sigma/\e,\quad
0<\e<\e_0.
\]
We only consider the problem in this time interval, in which \eqref{cd_sign} is satisfied.

\bigskip
\paragraph{\bf{5}}
We next prove an a priori estimate under the assumption that $R\leq0$ and $S\geq0$ in the characteristic triangle in Figure \ref{f0}, with $x_2-x_1<\e^{-\frac{2}{3}}$. 

Under this assumption, we have
\[
R_t-c R_x \geq - \frac{C_D}{4C_L}S^2-\frac{\mu}2 S.
\]
Integrating it along the backward characteristic by \eqref{dependence}, we have
\ben
R(t_0,x_0)&\geq& - \frac{C_D}{4C_L}\int_0^{t_0} S^2(t,x_{-}(t))dt-\frac{\mu}2 \int_0^{t_0} S(t,x_{-}(t))dt+R(0, x_1)\nn\\
&=&-\frac{C_D}{4C_L}\int_{x_0}^{x_2} \frac{S^2(t_{-}(x),x)}{c(t_{-}(x),x)}dx-\frac{\mu}2
\int_{x_0}^{x_2} \frac{S(t_{-}(x),x)}{c(t_{-}(x),x)}dx +R(0, x_1)\nn\\
&\geq& -\frac{C_D}{4C_L^2}E(0)-\frac{\mu}{2C_L}\sqrt{E(0)}\sqrt{|x_2-x_0|}+R(0, x_1)\nn\\
&\geq& -M_2 \e^{\frac1 6},\nn
\een
for some positive constant $M_2$.
Choosing $\e<(\frac{2\mu C_L}{M_2 C_N})^6$, we have 
\[ 
-\frac{c'(u)}{4c(u)} R^2-\frac{\mu}{2} R\geq0
\] 
because $R\leq 0$ under the a priori assumption we assumed.
So
\[
S_t+c S_x \geq \frac{c'(u)}{4c(u)} S^2-\frac{\mu}{2} S.
\]
\bigskip
\paragraph{\bf{6}}
Then we use this a priori estimate to prove that $R<0$ and $S>0$ in the closed region bounded by
the forward characteristic starting from the origin on $(t,x)$-plane, 
$t=\sigma /\e$, $t=0$ and the backward characteristic starting form the point $(0,\, \e^{-\frac{2}{3}}-\e^2)$ on $(t,x)$-plane. 
Here, we choose $\e^{-\frac{2}{3}}-\e^2$ instead of  $\e^{-\frac{2}{3}}$ to exclude the backward characteristic 
starting from the point $(0,\, \e^{-\frac{2}{3}})$. 

We prove it by contradiction. Assume that 
$R=0$ or $S=0$ at some point in the region we consider. By \eqref{RSIC_sing}, $R(0,x)<0$ and $S(0,x)>0$ when $x\in[0,\, \e^{-\frac{2}{3}}]$, so we can find the lowest line $t=t_1>0$ such that $R=0$ or 
$S=0$ on some point of this line, while $R<0$ and $S>0$ when $0\leq t<t_1$. So we can find a closed
characteristic triangle with $R=0$ or $S=0$ on the upper vertex while $R<0$ and $S>0$ elsewhere. Then we derive a contradiction by proving that $R$ and $S$ cannot be zero on the upper vertex.

Since $S\geq 0$ in the characteristic triangle,
\[
R_t-c R_x \leq\frac{c'(u)}{4c(u)} R^2-\frac{\mu}{2} R.
\]
Then by standard ODE comparison theorem (see \cite{Mc} for reference), \eqref{sing_bounds} and $R<0$ except the 
upper vertex, we have $R<0$ on the vertex.
By $S\geq0$ and $R\leq0$ in the characteristic triangle, the a priori estimate in the previous part shows that
\[
S_t+c S_x \geq \frac{c'(u)}{4c(u)} S^2-\frac{\mu}{2} S.
\]
Then by ODE comparison theorem, \eqref{sing_bounds} and $S>0$ except the upper vertex, we have $S>0$ on the vertex.
Hence, we get a contradiction. 
\bigskip
\paragraph{\bf{7}}
Finally, we prove the singularity formation by considering the forward characteristic starting from the origin.
We have already proved that along this characteristic, before 
$t=\frac\sigma \e$ and $t=\frac{\e^{-\frac{2}{3}}-\e^2}{2C_U}$,
\beq\label{sing_final}
S_t+c S_x \geq \frac{c'(u)}{4c(u)} S^2-\frac{\mu}{2} S\geq \frac{c'(u_0)}{8C_U} S^2-\frac{\mu}{2} S\geq \frac{c'(u_0)}{16C_U} S^2-\frac{\mu}{2} S.
\eeq

When $\e<C_L$, \eqref{S_0_sing} and the initial conditions \eqref{id_sing1} and \eqref{id_sing2} give that 
\beq\label{s00}
-C_U \phi'(0)>S(0,0)=-c\bigl(u(0,0)\bigr)\phi'(0)>\frac{8\mu C_U}{c'(u_0)}
\eeq
which shows
\[
\frac{c'(u_0)}{16C_U} S^2(0,0)>\frac{\mu}{2} S(0,0).
\]
By the ODE comparison theorem, on the characteristic, \eqref{sing_final} gives
\[
S(t,x)>\frac{8\mu C_U}{c'(u_0)},
\]
i.e.
\[
\frac{c'(u_0)}{16C_U} S^2(t,x)>\frac{\mu}{2} S(t,x).
\]
Hence, by \eqref{sing_final},
\[
S_t+c S_x \geq \frac{c'(u_0)}{16C_U} S^2.
\]
So $S$ blowups before $t=\frac{16C_U}{c'(u_0) S(0,0)}$ which is in $O(1)$ by \eqref{s00}. 
We complete the proof of Theorem \ref{sing}.
\section{Systems in energy-dependent coordinates \label{section_2}}
\setcounter{equation}{0}
In this section, by restricting our consideration on smooth solutions,
we derive a semi-linear system on new coordinates. 
In the next two sections, we construct the solution for the new semi-linear system, 
then after the reverse transformation we show that this solution is also a weak solution 
of the original system. 
\subsection{Energy-dependent coordinates}
We denote 
\beq 
\vec{R}=(R_1,R_2,R_3)\eqdefa
\n_t+c(n_1)\n_x,\quad \vec{S}=(S_1,S_2,S_3)\eqdefa
\n_t-c(n_1)\n_x.\label{2.1} 
\eeq 
Without confusion, we still use the letters $R$ and $S$ to denote Riemann variables as in the previous section.
Then (\ref{1.1}) can be reformulated as: 
\beq \left\{
\begin{array}{ll}
\p_tR_1-c(n_1)\p_xR_1=\f1{4c^2(n_1)}\bigl\{(c^2(n_1)-\ga)(|\vec{R}|^2+|\vec{S}|^2)-2(3c^2(n_1)-\ga)\vec{R}\cdot\vec{S}\bigr\}n_1\\
\qquad\qquad\qquad\qquad\quad+\f{c'(n_1)}{2c(n_1)}(R_1-S_1)R_1-\frac{\mu}{2}(R_1+S_1),\\
\p_tS_1+c(n_1)\p_xS_1=\f1{4c^2(n_1)}\bigl\{(c^2(n_1)-\ga)(|\vec{R}|^2+|\vec{S}|^2)-2(3c^2(n_1)-\ga)\vec{R}\cdot\vec{S}\bigr\}n_1\\
\qquad\qquad\qquad\qquad\quad-\f{c'(n_1)}{2c(n_1)}(R_1-S_1)S_1-\frac{\mu}{2}(R_1+S_1)\\
\p_tR_2-c(n_1)\p_xR_2=\f1{4c^2(n_1)}\bigl\{(c^2(n_1)-\al)(|\vec{R}|^2+|\vec{S}|^2)-2(3c^2(n_1)-\al)\vec{R}\cdot\vec{S}\bigr\}n_2\\
\qquad\qquad\qquad\qquad\quad+\f{c'(n_1)}{2c(n_1)}(R_2-S_2)R_1-\frac{\mu}{2}(R_2+S_2), \\
\p_tS_2+c(n_1)\p_xS_2=\f1{4c^2(n_1)}\bigl\{(c^2(n_1)-\al)(|\vec{R}|^2+|\vec{S}|^2)-2(3c^2(n_1)-\al)\vec{R}\cdot\vec{S}\bigr\}n_2\\
\qquad\qquad\qquad\qquad\quad-\f{c'(n_1)}{2c(n_1)}(R_2-S_2)S_1-\frac{\mu}{2}(R_2+S_2), \\
\p_tR_3-c(n_1)\p_xR_3=\f1{4c^2(n_1)}\bigl\{(c^2(n_1)-\al)(|\vec{R}|^2+|\vec{S}|^2)-2(3c^2(n_1)-\al)\vec{R}\cdot\vec{S}\bigr\}n_3\\
\qquad\qquad\qquad\qquad\quad+\f{c'(n_1)}{2c(n_1)}(R_3-S_3)R_1-\frac{\mu}{2}(R_3+S_3), \\
\p_tS_3+c(n_1)\p_xS_3=\f1{4c^2(n_1)}\bigl\{(c^2(n_1)-\al)(|\vec{R}|^2+|\vec{S}|^2)-2(3c^2(n_1)-\al)\vec{R}\cdot\vec{S}\bigr\}n_3\\
\qquad\qquad\qquad\qquad\quad-\f{c'(n_1)}{2c(n_1)}(R_3-S_3)S_1-\frac{\mu}{2}(R_3+S_3), \\
\n_x=\frac{\vec{R}-\vec{S}}{2c(n_1)}\quad \mbox{or} \quad \n_t=\frac{\vec{R}+\vec{S}}2.\\
  \end{array}\right. \label{2.2} \eeq
The energy equation \eqref{1.4a2} equals to
\beq\label{2.3}
\f14\p_t\bigl(|\vec{R}|^2+|\vec{S}|^2\bigr)-\f14\p_x\bigl[c(n_1)(|\vec{R}|^2-|\vec{S}|^2)\bigr]=-\frac{\mu}{4}|\vec{R}+\vec{S}|^2\leq 0.
\eeq

We define the forward and backward characteristics as follows \beq
\left\{
\begin{array}{ll}
\frac{d}{ds}x^\pm(s,t,x)=\pm c(n_1(s,x^\pm(s,t,x))),\\
x^\pm|_{s=t}=x.
\end{array}\right. \label{2.4} \eeq
Then we define the coordinate transformation: 
\beno 
X\eqdefa
\int_0^{x^-(0,t,x)}[1+|\vec{R}|^2(0,y)]\,dy,\quad\mbox{and}\quad
Y\eqdefa \int_{x^+(0,t,x)}^0[1+|\vec{S}|^2(0,y)]\,dy. \eeno This
implies \beq X_t-c(n_1)X_x=0,\quad Y_t+c(n_1)Y_x=0. \label{2.6} \eeq
Furthermore, for any smooth function $f,$ we get by using
(\ref{2.6}) that 
\beq\begin{split}
&f_t+c(n_1)f_x=(X_t+c(n_1)X_x)f_X=2c(n_1)X_xf_X\\
&f_t-c(n_1)f_x=(Y_t-c(n_1)Y_x)f_Y=-2c(n_1)Y_xf_Y.
\end{split} \label{2.6a} \eeq 

Using (\ref{2.6a}), we can transform the directional derivatives $\partial_t\pm c \partial_x$
on the $(t,x)$-coordinates into $\partial_X$ and $\partial_Y$ on the $(X,Y)$-coordinates,
hence get a new semi-linear system on $(X,Y)$-coordinates.
We start the calculation from $u_X$ and $u_Y$. 
To complete the system, we introduce several new variables:
\beq p\eqdefa
\frac{1+|\vec{R}|^2}{X_x},\qquad q\eqdefa
\frac{1+|\vec{S}|^2}{-Y_x}, \label{2.6b} \eeq 
and
\beq\begin{split}
&\vec{\ell}=(\ell_1,\ell_2,\ell_3)\eqdefa\frac{\vec{R}}{1+|\vec{R}|^2},\quad
\vec{m}=(m_1,m_2,m_3)\eqdefa\frac{\vec{S}}{1+|\vec{S}|^2},\quad
\mbox{and}\\
& h_1\eqdefa \frac1{1+|\vec{R}|^2},\quad h_2\eqdefa
\frac1{1+|\vec{S}|^2}.\end{split} \label{2.11} 
\eeq
For smooth solutions, $|\vec\ell|$, $|\vec{m}|$, $|h_1|$ and $|h_2|$ are all less than $1$. 
However, $|p|$ and $|q|$ might go to infinity when the possible gradient blowup happens.
In the next section, in order to estimating them, we will consider the energy equation, where the energy equation can 
help us estimate  $p$ and $q$ since 
$X$ and $Y$ are {\emph{energy-dependent}} coordinates. 

We derive the new system in several steps:
\bigskip

\paragraph{\bf{1}}
Using (\ref{2.6a}) and \eqref{2.11}, we have
\beq\label{2.18}\begin{split}
\p_Y\n=&\frac1{2c(n_1)(-Y_x)}(\p_t\n-c(n_1)\p_x\n)=\f{q}{2c(n_1)}\vec{m},\\
\p_X\n=&\frac1{2c(n_1)X_x}(\p_t\n+c(n_1)\p_x\n)=\f{p}{2c(n_1)}\vec{\ell}.
\end{split}
\eeq

\bigskip
\paragraph{\bf{2}}
Then \beno
\p_tp-c(n_1)\p_xp&=&2(X_x)^{-1}[\vec{R}\cdot(\p_t\vec{R}-c(n_1)\p_x\vec{R})]\\
&&-(X_x)^{-2}[\p_tX_x-c(n_1)\p_xX_x](1+|\vec{R}|^2). \eeno While
using (\ref{2.2}), we have \beno
\begin{split}
\vec{R}\cdot(\p_t\vec{R}-c(n_1)\p_x\vec{R})
=\frac{1}{4c^2(n_1)}\bigl\{(c^2(n_1)-\al)(|\vec{R}|^2+|\vec{S}|^2)-2(3c^2(n_1)-\al)\vec{R}\cdot\vec{S}\,\bigr\}\vec{R}\cdot\n\\
\qquad
+\f{\al-\ga}{4c^2(n_1)}(|\vec{R}|^2+|\vec{S}|^2-2\vec{R}\cdot\vec{S})R_1n_1+\f{c'(n_1)}{2c(n_1)}R_1(|\vec{R}|^2-\vec{R}\cdot\vec{S}) -\frac{\mu}{2}(|\vec R|^2+\vec{R}\cdot\vec{S}),
\end{split}
\eeno while notice that $|\n|=1$ and
$c'(n_1)=\f{(\ga-\al)n_1}{c(n_1)},$ so that $\vec{R}\cdot\n=0,$ and
there holds \beq
 \label{2.7}
\vec{R}\cdot(\p_t\vec{R}-c(n_1)\p_x\vec{R})=\f{c'(n_1)}{4c(n_1)}R_1(|\vec{R}|^2-|\vec{S}|^2)
-\frac{\mu}{2}(|\vec R|^2+\vec{R}\cdot\vec{S}),
\eeq
which together with (\ref{2.6}) applied gives
\beno
\p_t p-c(n_1)\p_x p=\frac{p}{1+|\vec{R}|^2}\Big[\frac{c'(n_1)}{2c(n_1)}
\bigl(-R_1(1+|\vec{S}|^2)+S_1(1+|\vec{R}|^2)\bigr)  
-\mu(|\vec R|^2+\vec{R}\cdot\vec{S})\Big], 
\eeno 
from which, and
(\ref{2.6a}-\ref{2.6b}), we infer 
\beq\begin{split}
p_Y=&\frac1{2c(n_1)(-Y_x)}(p_t-c(n_1)p_x)\\
=&\frac{pq}{2c(n_1)}\Big[\frac{c'(n_1)}{2c(n_1)}(-\frac{R_1}{1+|\vec{R}|^2}+
\frac{S_1}{1+|\vec{S}|^2})-\mu\frac{|\vec R|^2+\vec{R}\cdot\vec{S}}
{(1+|\vec{S}|^2)(1+|\vec{R}|^2)}\Big].\end{split} \label{2.8} \eeq
Similarly, 
\beno
\p_t q+c(n_1)\p_x q&=&2(-Y_x)^{-1}[\vec{S}\cdot(\p_t\vec{S}+c(n_1)\p_x\vec{S})]\\
&&+(Y_x)^{-2}(\p_tY_x+c(n_1)\p_xY_x)(1+|\vec{S}|^2),\eeno 
then by (\ref{2.2}), we have 
\beno\begin{split}
\vec{S}\cdot(\p_t\vec{S}+c(n_1)\p_x\vec{S})
=\frac{1}{4c^2(n_1)}\bigl\{(c^2(n_1)-\al)(|\vec{R}|^2+|\vec{S}|^2)-2(3c^2(n_1)-\al)\vec{R}\cdot\vec{S}\,\bigr\}\vec{S}\cdot\n\\
\qquad
+\f{\al-\ga}{4c^2(n_1)}\bigl(|\vec{R}|^2+|\vec{S}|^2-2\vec{R}\cdot\vec{S}\bigr)S_1n_1+\f{c'(n_1)}{2c(n_1)}S_1(|\vec{S}|^2-\vec{R}\cdot\vec{S}) -\frac{\mu}{2}(|\vec S|^2+\vec{R}\cdot\vec{S}),
\end{split}
\eeno which along with the fact that $\vec{S}\cdot\n=0$ leads to
\beq
\vec{S}\cdot(\p_t\vec{S}+c(n_1)\p_x\vec{S})=-\f{c'(n_1)}{4c(n_1)}S_1(|\vec{R}|^2-|\vec{S}|^2) -\frac{\mu}{2}(|\vec S|^2+\vec{R}\cdot\vec{S}),
\label{2.9} \eeq 
from which  and (\ref{2.6}), we infer 
\beno
\p_t q+c(n_1)\p_x q=\frac{q}{1+|\vec{S}|^2}\Big[\frac{c'(n_1)}{2c(n_1)}
\bigl(R_1(1+|\vec{S}|^2)-S_1(1+|\vec{R}|^2)\bigr)  
-\mu(|\vec S|^2+\vec{R}\cdot\vec{S})\Big].
\eeno 
Next by
(\ref{2.6a}-\ref{2.6b}), we have 
\beq\begin{split} 
q_X=&\frac1{2c(n_1)X_x}(q_t+c(n_1)q_x)\\
=&\frac{pq}{2c(n_1)}\Big[\frac{c'(n_1)}{2c(n_1)}\bigl(\frac{R_1}{1+|\vec{R}|^2}-
\frac{S_1}{1+|\vec{S}|^2}\bigr)-\mu\frac{|\vec S|^2+\vec{R}\cdot\vec{S}}
{(1+|\vec{S}|^2)(1+|\vec{R}|^2)}\Big].
\end{split}
\label{2.10} \eeq 
\paragraph{\bf{3}}
Next, \beno
\p_t\ell_1-c(n_1)\p_x\ell_1&=&(1+|\vec{R}|^2)^{-1}(\p_tR_1-c(n_1)\p_xR_1)\\
&&-2(1+|\vec{R}|^2)^{-2}R_1\bigl[\vec{R}\cdot(\p_t\vec{R}-c(n_1)\p_x\vec{R})\,\bigr],
 \eeno which
together with (\ref{2.2}) and (\ref{2.7}) implies 
\beno\begin{split}
\p_t\ell_1-c(n_1)\p_x\ell_1=(1+|\vec{R}|^2)^{-1}\bigl[\f{c^2(n_1)-\ga}{4c^2(n_1)}(|\vec{R}|^2+|\vec{S}|^2)
-\f{3c^2(n_1)-\ga}{2c^2(n_1)}\vec{R}\cdot\vec{S}\,\bigr]n_1\\
\qquad+
(1+|\vec{R}|^2)^{-2}R_1\frac{c'(n_1)}{2c(n_1)}\big[(R_1(1+|\vec{S}|^2)-S_1(1+|\vec{R}|^2)\big]\\
+(1+|\vec{R}|^2)^{-2}R_1\mu(|\vec R|^2+\vec{R}\cdot\vec{S})
-(1+|\vec{R}|^2)^{-1}\frac{\mu}{2}(R_1+S_1).
\end{split}
\eeno Using (\ref{2.6a}-\ref{2.6b}), we obtain
\beq\begin{split}
\p_Y\ell_1&=\f{q}{8c^3(n_1)}\bigl[(c^2(n_1)-\ga)(h_1+h_2-2h_1h_2)-2(3c^2(n_1)-\ga)\vec{\ell}\cdot\vec{m}\bigr]n_1\\
&+\frac{c'(n_1)}{4c^2(n_1)}\ell_1q(\ell_1-m_1) +\frac{\mu q}{4c(n_1)}\bigl[2 \ell_1(h_2-h_1 h_2+\vec{\ell}\cdot \vec{m})-\ell_1 h_2-m_1 h_1\bigr].
\end{split} \label{2.12} 
\eeq 
Similarly by (\ref{2.11}), we have
\beno
\begin{split}
\p_tm_1+c(n_1)\p_xm_1=&(1+|\vec{S}|^2)^{-1}(\p_tS_1+c(n_1)\p_xS_1)\\
&-2(1+|\vec{S}|^2)^{-2}S_1\bigl[\vec{S}\cdot(\p_t\vec{S}+c(n_1)\p_x\vec{S})\,\bigr],
\end{split}\eeno which together with \eqref{2.2} and (\ref{2.9}) ensures that \beno\begin{split}
\p_tm_1+c(n_1)\p_xm_1=(1+|\vec{S}|^2)^{-1}\bigl[\f{c^2(n_1)-\ga}{4c^2(n_1)}(|\vec{R}|^2+|\vec{S}|^2)
-\f{3c^2(n_1)-\ga}{2c^2(n_1)}\vec{R}\cdot\vec{S}\,\bigr]n_1\\
\qquad-
(1+|\vec{S}|^2)^{-2}S_1\frac{c'(n_1)}{2c(n_1)}\big[(R_1(1+|\vec{S}|^2)-S_1(1+|\vec{R}|^2)\big]
\\
+(1+|\vec{S}|^2)^{-2}S_1\mu(|\vec S|^2+\vec{R}\cdot\vec{S})
-(1+|\vec{S}|^2)^{-1}\frac{\mu}{2}(R_1+S_1),
\end{split}
\eeno  from which and (\ref{2.6a}-\ref{2.6b}), we infer
\beq\begin{split}
\p_Xm_1&=\f{p}{8c^3(n_1)}\bigl[(c^2(n_1)-\ga)(h_1+h_2-2h_1h_2)-2(3c^2(n_1)-\ga)\vec{\ell}\cdot\vec{m}\bigr]n_1\\
&-\frac{c'(n_1)}{4c^2(n_1)}m_1p(\ell_1-m_1)+\frac{\mu p}{4c(n_1)}\bigl[2 m_1(h_1-h_1 h_2+\vec{\ell}\cdot \vec{m})-\ell_1 h_2-m_1 h_1\bigr].
\end{split} \label{2.13} \eeq
Following  the same line, we deduce from  \eqref{2.2} and
\eqref{2.7} that \beno
\begin{split}
\p_t\ell_2-c(n_1)\p_x\ell_2=(1+|\vec{R}|^2)^{-1}\bigl[\f{c^2(n_1)-\al}{4c^2(n_1)}(|\vec{R}|^2+|\vec{S}|^2)
-\f{3c^2(n_1)-\al}{2c^2(n_1)}\vec{R}\cdot\vec{S}\,\bigr]n_2\\
+
(1+|\vec{R}|^2)^{-2}R_1\frac{c'(n_1)}{2c(n_1)}\bigl[R_2(1+|\vec{S}|^2)-S_2(1+|\vec{R}|^2)\bigl]\\
+(1+|\vec{R}|^2)^{-2}R_2\mu(|\vec R|^2+\vec{R}\cdot\vec{S})
-(1+|\vec{R}|^2)^{-1}\frac{\mu}{2}(R_2+S_2),
\end{split}
\eeno and \beno
\begin{split}
\p_t\ell_3-c(n_1)\p_x\ell_3=(1+|\vec{R}|^2)^{-1}\bigl[\f{c^2(n_1)-\al}{4c^2(n_1)}(|\vec{R}|^2+|\vec{S}|^2)
-\f{3c^2(n_1)-\al}{2c^2(n_1)}\vec{R}\cdot\vec{S}\,\bigr]n_3\\
+
(1+|\vec{R}|^2)^{-2}R_1\frac{c'(n_1)}{2c(n_1)}\bigl[R_3(1+|\vec{S}|^2)-S_3(1+|\vec{R}|^2)\bigl]\\
+(1+|\vec{R}|^2)^{-2}R_3\mu(|\vec R|^2+\vec{R}\cdot\vec{S})
-(1+|\vec{R}|^2)^{-1}\frac{\mu}{2}(R_3+S_3),
\end{split}
\eeno which together with (\ref{2.6a}-\ref{2.6b}) derive that \beq
\begin{split}
\p_Y\ell_2&=\f{q}{8c^3(n_1)}\bigl[(c^2(n_1)-\al)(h_1+h_2-2h_1h_2)-2(3c^2(n_1)-\al)\vec{\ell}\cdot\vec{m}\bigr]n_2\\
&+\frac{c'(n_1)}{4c^2(n_1)}\ell_1 q(\ell_2-m_2)+\frac{\mu q}{4c(n_1)}\bigl[2 \ell_2(h_2-h_1 h_2+\vec{\ell}\cdot \vec{m})-\ell_2 h_2-m_2 h_1\bigr],\\
\p_Y\ell_3&=\f{q}{8c^3(n_1)}\bigl[(c^2(n_1)-\al)(h_1+h_2-2h_1h_2)-2(3c^2(n_1)-\al)\vec{\ell}\cdot\vec{m}\bigr]n_3\\
&+\frac{c'(n_1)}{4c^2(n_1)}\ell_1 q(\ell_3-m_3)+\frac{\mu q}{4c(n_1)}\bigl[2 \ell_3(h_2-h_1 h_2+\vec{\ell}\cdot \vec{m})-\ell_3 h_2-m_3 h_1\bigr].
\end{split} \label{2.14} \eeq
While we deduce from \eqref{2.2} and \eqref{2.9} that \beno
\begin{split}
\p_tm_2+c(n_1)\p_xm_2=(1+|\vec{S}|^2)^{-1}\bigl[\f{c^2(n_1)-\al}{4c^2(n_1)}(|\vec{R}|^2+|\vec{S}|^2)
-\f{3c^2(n_1)-\al}{2c^2(n_1)}\vec{R}\cdot\vec{S}\,\bigr]n_2\\
-
(1+|\vec{S}|^2)^{-2}S_1\frac{c'(n_1)}{2c(n_1)}\bigl[R_2(1+|\vec{S}|^2)-S_2(1+|\vec{R}|^2)\bigl]
\\
+(1+|\vec{S}|^2)^{-2}S_2\mu(|\vec S|^2+\vec{R}\cdot\vec{S})
-(1+|\vec{S}|^2)^{-1}\frac{\mu}{2}(R_2+S_2),
\end{split}\eeno
and \beno
\begin{split}
\p_tm_3+c(n_1)\p_xm_3=(1+|\vec{S}|^2)^{-1}\bigl[\f{c^2(n_1)-\al}{4c^2(n_1)}(|\vec{R}|^2+|\vec{S}|^2)
-\f{3c^2(n_1)-\al}{2c^2(n_1)}\vec{R}\cdot\vec{S}\,\bigr]n_3\\
-
(1+|\vec{S}|^2)^{-2}S_1\frac{c'(n_1)}{2c(n_1)}\bigl[R_3(1+|\vec{S}|^2)-S_3(1+|\vec{R}|^2)\bigl]
\\
+(1+|\vec{S}|^2)^{-2}S_3\mu(|\vec S|^2+\vec{R}\cdot\vec{S})
-(1+|\vec{S}|^2)^{-1}\frac{\mu}{2}(R_3+S_3),
\end{split}\eeno
which together with   (\ref{2.6a}-\ref{2.6b})  implies that
\beq\begin{split}
\p_Xm_2&=\f{p}{8c^3(n_1)}\bigl[(c^2(n_1)-\al)(h_1+h_2-2h_1h_2)-2(3c^2(n_1)-\al)
\vec{\ell}\cdot\vec{m}\bigr]n_2\\
&-\frac{c'(n_1)}{4c^2(n_1)}m_1p(\ell_2-m_2)+\frac{\mu p}{4c(n_1)}\bigl[2 m_2(h_1-h_1 h_2
+\vec{\ell}\cdot \vec{m})-\ell_2 h_2-m_2 h_1\bigr],\\
\p_Xm_3&=\f{p}{8c^3(n_1)}\bigl[(c^2(n_1)-\al)(h_1+h_2-2h_1h_2)-2(3c^2(n_1)-\al)
\vec{\ell}\cdot\vec{m}\bigr]n_3\\
&-\frac{c'(n_1)}{4c^2(n_1)}m_1p(\ell_3-m_3)+\frac{\mu p}{4c(n_1)}\bigl[2 m_3(h_1-h_1 h_2
+\vec{\ell}\cdot \vec{m})-\ell_3 h_2-m_3 h_1\bigr].
\end{split} \label{2.15} \eeq
\bigskip

\paragraph{\bf{4}}
It follows from
(\ref{2.7}) and \eqref{2.11} that 
\beno
\begin{split}
\p_th_1-c(n_1)\p_xh_1=&-2(1+|\vec{R}|^2)^{-2}\bigl[\vec{R}\cdot(\p_t\vec{R}-c(n_1)\p_x\vec{R})\bigl]\\
=&-(1+|\vec{R}|^2)^{-2}\bigl[\f{c'(n_1)}{2c(n_1)}R_1(|\vec{R}|^2-|\vec{S}|^2)
-\mu(|\vec R|^2+\vec{R}\cdot\vec{S})\bigl].
\end{split} 
\eeno 
Then we get by using (\ref{2.6a}-\ref{2.6b}) that  
\beq
\p_Yh_1=\frac{c'(n_1)}{4c^2(n_1)}q\ell_1(h_1-h_2) +\frac{\mu}{2c(n_1)} q h_1(h_2-h_1 h_2
+\vec{\ell}\cdot \vec{m}). \label{2.16} 
\eeq
Similar calculations together with (\ref{2.9}) gives 
\beno
\begin{split}
\p_th_2+c(n_1)\p_xh_2=(1+|\vec{S}|^2)^{-2}\big[\frac{c'(n_1)}{2c(n_1)}S_1(|\vec{R}|^2-|\vec{S}|^2)
+\mu(|\vec S|^2+\vec{R}\cdot\vec{S})\big],
\end{split} \eeno 
which together with  (\ref{2.6a}-\ref{2.6b}) implies that 
\beq 
\p_Xh_2=\frac{c'(n_1)}{4c^2(n_1)}pm_1(h_2-h_1)+\frac{\mu}{2c(n_1)} p h_2(h_1-h_1 h_2+\vec{\ell}\cdot \vec{m}).
\label{2.17} 
\eeq 
\bigskip

\paragraph{\bf{5}}
In summary, we obtain \beq \left\{
\begin{array}{l}
\p_Y\ell_1=\f{q}{8c^3(n_1)}\bigl[(c^2(n_1)-\ga)(h_1+h_2-2h_1h_2)-2(3c^2(n_1)-\ga)\vec{\ell}\cdot\vec{m}\bigr]n_1\\
\qquad\qquad+\frac{c'(n_1)}{4c^2(n_1)}\ell_1q(\ell_1-m_1) +\frac{\mu q}{4c(n_1)}\bigl[2 \ell_1(h_2-h_1 h_2+\vec{\ell}\cdot \vec{m})-\ell_1 h_2-m_1 h_1\bigr], \\
\p_Xm_1=\f{p}{8c^3(n_1)}\bigl[(c^2(n_1)-\ga)(h_1+h_2-2h_1h_2)-2(3c^2(n_1)-\ga)\vec{\ell}\cdot\vec{m}\bigr]n_1\\
\qquad\qquad-\frac{c'(n_1)}{4c^2(n_1)}m_1p(\ell_1-m_1)+\frac{\mu p}{4c(n_1)}\bigl[2 m_1(h_1-h_1 h_2+\vec{\ell}\cdot \vec{m})-\ell_1 h_2-m_1 h_1\bigr],\\
\p_Y\ell_2=\f{q}{8c^3(n_1)}\bigl[(c^2(n_1)-\al)(h_1+h_2-2h_1h_2)-2(3c^2(n_1)-\al)\vec{\ell}\cdot\vec{m}\bigr]n_2\\
\qquad\qquad+\frac{c'(n_1)}{4c^2(n_1)}\ell_1 q(\ell_2-m_2)+\frac{\mu q}{4c(n_1)}\bigl[2 \ell_2(h_2-h_1 h_2+\vec{\ell}\cdot \vec{m})-\ell_2 h_2-m_2 h_1\bigr],\\
\p_Xm_2=\f{p}{8c^3(n_1)}\bigl[(c^2(n_1)-\al)(h_1+h_2-2h_1h_2)-2(3c^2(n_1)-\al)\vec{\ell}\cdot\vec{m}\bigr]n_2\\
\qquad\qquad-\frac{c'(n_1)}{4c^2(n_1)}m_1p(\ell_2-m_2)+\frac{\mu p}{4c(n_1)}\bigl[2 m_2(h_1-h_1 h_2+\vec{\ell}\cdot \vec{m})-\ell_2 h_2-m_2 h_1\bigr],\\
\p_Y\ell_3=\f{q}{8c^3(n_1)}\bigl[(c^2(n_1)-\al)(h_1+h_2-2h_1h_2)-2(3c^2(n_1)-\al)\vec{\ell}\cdot\vec{m}\bigr]n_3\\
\qquad\qquad+\frac{c'(n_1)}{4c^2(n_1)}\ell_1 q(\ell_3-m_3)+\frac{\mu q}{4c(n_1)}\bigl[2 \ell_3(h_2-h_1 h_2+\vec{\ell}\cdot \vec{m})-\ell_3 h_2-m_3 h_1\bigr],\\
\p_Xm_3=\f{p}{8c^3(n_1)}\bigl[(c^2(n_1)-\al)(h_1+h_2-2h_1h_2)-2(3c^2(n_1)-\al)\vec{\ell}\cdot\vec{m}\bigr]n_3\\
\qquad\qquad-\frac{c'(n_1)}{4c^2(n_1)}m_1p(\ell_3-m_3)+\frac{\mu p}{4c(n_1)}\bigl[2 m_3(h_1-h_1 h_2+\vec{\ell}\cdot \vec{m})-\ell_3 h_2-m_3 h_1\bigr],\\
\p_Y\n=\f{q}{2c(n_1)}\vec{m},\quad\mbox{or}\quad
\p_X\n=\f{p}{2c(n_1)}\vec{\ell},\\
\p_Yh_1=\frac{c'(n_1)}{4c^2(n_1)}q\ell_1(h_1-h_2) +\frac{\mu}{2c(n_1)} q h_1(h_2-h_1 h_2+\vec{\ell}\cdot \vec{m}),\\
\p_Xh_2=\frac{c'(n_1)}{4c^2(n_1)}pm_1(h_2-h_1)+\frac{\mu}{2c(n_1)} p h_2(h_1-h_1 h_2+\vec{\ell}\cdot \vec{m}),\\
p_Y=\frac{pq}{2c(n_1)}\bigl[-\frac{c'(n_1)}{2c(n_1)}\big( \ell_1-m_1\big)-\mu (h_2-h_1h_2+\vec{\ell}\cdot \vec{m})\bigr],\\
q_X=\frac{pq}{2c(n_1)}\bigl[\frac{c'(n_1)}{2c(n_1)}\big( \ell_1-m_1\big)-\mu (h_1-h_1h_2+\vec{\ell}\cdot \vec{m})\bigr].
\end{array}\right. \label{2.19} \eeq
\subsection{Consistency of variables}
Before we prove the global existence of \eqref{2.19},
we first show that the various variables introduced for \eqref{2.19} are consistent, following from the proposition below.
Note we only use \eqref{2.19} in the proof of this proposition, while the original equation (\ref{1.1}) and definitions (\ref{2.6b}) and (\ref{2.11}) are not used.

\begin{prop}\label{prop2.1}
{\sl For smooth enough data, the following conservative
quantities hold: \beq\label{2.20}
\begin{split}
&\vec{\ell}\cdot\n(X,Y)=\vec{m}\cdot\n(X,Y)=0\quad |\n(X,Y)|=1\quad
\mbox{and}\\
& |\vec{\ell}(X,Y)|^2+h_1^2(X,Y)=h_1(X,Y),\quad
|\vec{m}(X,Y)|^2+h_2^2(X,Y)=h_2(X,Y) \quad\forall\ X, Y,
\end{split}
\eeq as long as \beno
\begin{split}
&\vec{\ell}\cdot\n(X,\varphi(X))=\vec{m}\cdot\n(X,\varphi(X))=0\quad
|\n(X,\varphi(X))|=1\quad
\mbox{and}\\
&
|\vec{\ell}(X,\varphi(X))|^2+h_1^2(X,\varphi(X))=h_1(X,\varphi(X)),\\
&|\vec{m}(X,\varphi(X))|^2+h_2^2(X,\varphi(X))=h_2(X,\varphi(X)).
\end{split}
\eeno }
Here $\varphi$ represents the initial curve, see \eqref{1.33} from the next section.
\end{prop}

\begin{proof} We first deduce from \eqref{2.19} that
\beno
\begin{split}
\p_Y\bigl[ |\vec{\ell}|^2+h_1^2-h_1\bigr]=&\f{q}{4c^3(n_1)}\bigl[(c^2(n_1)-\al)(h_1+h_2-2h_1h_2)-2(3c^2(n_1)-\al)\vec{\ell}\cdot\vec{m}\bigr]\vec{\ell}\cdot\n\\
& +\frac{(\al-\ga)n_1}{4c^3(n_1)}q\bigl[(h_1+h_2-2h_1h_2)-2\vec{\ell}\cdot\vec{m}\bigr]\ell_1\\
&+\f{c'(n_1)}{4c^2(n_1)}q\ell_1\bigl[2(|\vec{\ell}|^2-\vec{\ell}\cdot\vec{m})+2h_1(h_1-h_2)-(h_1-h_2)\bigr]\\
&+\frac{\mu q}{2c(n_1)}\vec\ell\cdot\bigl[2\vec\ell(h_2-h_1h_2+\vec{\ell}\cdot\vec{m})-\vec{\ell} h_2 -\vec{m}h_1 \bigr]\\
&+\frac{\mu q}{2c(n_1)} (2 h_1^2-h_1)(h_2-h_1h_2+\vec{\ell}\cdot\vec{m}),
\end{split}
\eeno from which and the fact that
$c'(n_1)=\f{(\ga-\al)n_1}{c(n_1)},$ we deduce that \beq\label{2.21}
\begin{split}
\p_Y\bigl[ |\vec{\ell}|^2+h_1^2-h_1\bigr]
=\f{q}{4c^3(n_1)}\bigl[(c^2(n_1)-\al)(h_1+h_2-2h_1h_2)
-2(3c^2(n_1)-\al)\vec{\ell}\cdot\vec{m}\bigr]\vec{\ell}\cdot\n\\
+\bigl(\f{c'(n_1)}{2c^2(n_1)}q\ell_1+ \frac{\mu q}{2 c(n_1)}[h_2-2h_1h_2+2\vec{\ell}\cdot\vec{m}]\bigr)\bigl[
|\vec{\ell}|^2+h_1^2-h_1\bigr].
\end{split}
\eeq Similarly it follows from \eqref{2.19} that \beno
\begin{split}
\p_Y[\vec{\ell}\cdot\n]=&\p_Y\vec{\ell}\cdot\n+\vec{\ell}\cdot\p_Y\n\\
=&\f{q}{8c^3(n_1)}\bigl[(c^2(n_1)-\al)(h_1+h_2-2h_1h_2)-2(3c^2(n_1)-\al)\vec{\ell}\cdot\vec{m}\bigr](|\n|^2-1)\\
&+\f{q}{8c^3(n_1)}(h_1+h_2-2h_1h_2)\bigl[(c^2(n_1)-\al)+(\al-\ga)n_1^2\bigr]\\
&-\f{q}{4c^3(n_1)}\vec{\ell}\cdot\vec{m}\bigl[3c^2(n_1)-\al+(\al-\ga)n_1^2\bigr]\\
&+\f{q}{4c^2(n_1)}q\ell_1[\vec{\ell}\cdot\n-\vec{m}\cdot\n]+\f{q}{2c(n_1)}\vec{\ell}\cdot\vec{m}\\
&+\frac{\mu q}{4c(n_1)}\bigl[\vec\ell(h_2-2h_1h_2+2\vec{\ell}\cdot\vec{m}) -\vec{m}h_1 \bigr]\cdot\n,
\end{split}
\eeno which along with the fact that $c^2(n_1)=\al+(\ga-\al)n_1^2$
leads to \beq \label{2.22}
\begin{split}
\p_Y[\vec{\ell}\cdot\n]=\f{q}{8c^3(n_1)}\bigl[(c^2(n_1)-\al)(h_1+h_2-2h_1h_2)-2(3c^2(n_1)-\al)\vec{\ell}\cdot\vec{m}\bigr](|\n|^2-1)\\
\qquad+\f{c'(n_1)}{4c^2(n_1)}q\ell_1[\vec{\ell}\cdot\n-\vec{m}\cdot\n] +\frac{\mu q}{4c(n_1)}(h_2-2h_1h_2+2\vec{\ell}\cdot\vec{m})\, (\vec\ell\cdot\n) -\frac{\mu q h_1}{4c(n_1)}\,(\vec{m}\cdot\n).
\end{split}
\eeq On the other hand, observe that $\p_Y\n=\f{q}{2c(n_1)}\vec{m}$
is consistent with $\p_X\n=\f{p}{2c(n_1)}\vec{\ell}.$In fact, again
thanks to \eqref{2.19}, one has on the one hand \beno
\p_X[\p_Y\n]=-\f{c'(n_1)}{8c^3(n_1)}pq(\ell_1+m_1)\vec{m}-\f{\mu}{4c^2(n_1)}pq(h_1-h_1h_2+\vec{\ell}\cdot \vec{m})\vec{m}+\f{q}{2c(n_1)}\p_X\vec{m},
\eeno and on the other hand \beno
\p_Y[\p_X\n]=-\f{c'(n_1)}{8c^3(n_1)}pq(\ell_1+m_1)\vec{\ell}-\f{\mu}{4c^2(n_1)}pq (h_2-h_1h_2+\vec{\ell}\cdot \vec{m})\vec{\ell}+\f{p}{2c(n_1)}\p_Y\vec{\ell},
\eeno 
which along with the $\vec{\ell}$ equations and $\vec{m}$
equations of \eqref{2.19} shows that $\p_X[\p_Y\n]=\p_Y[\p_X\n].$ So
we can also use the equation  $\p_X\n=\f{p}{2c(n_1)}\vec{\ell},$
from which and the $\vec{m}$ equation of \eqref{2.19}, we get
\beno\begin{split}
\p_X[\vec{m}\cdot\n]=&\p_X\vec{m}\cdot\n+\vec{m}\cdot\p_X\n\\
=&\f{p}{8c^3(n_1)}\bigl[(c^2(n_1)-\al)(h_1+h_2-2h_1h_2)-2(3c^2(n_1)-\al)\vec{\ell}\cdot\vec{m}\bigr](|\n|^2-1)\\
&+\f{p}{8c^3(n_1)}\bigl[(c^2(n_1)-\al)(h_1+h_2-2h_1h_2)-2(3c^2(n_1)-\al)\vec{\ell}\cdot\vec{m}\bigr]\\
&+\f{\al-\ga}{8c^3(n_1)}pn_1^2\bigl[(h_1+h_2-2h_1h_2)-2\vec{\ell}\cdot\vec{m}\bigr]\\
&+\f{c'(n_1)}{4c^2(n_1)}pm_1[\vec{m}\cdot\n-\vec{\ell}\cdot\n]+\f{p}{2c(n_1)}\vec{\ell}\cdot\vec{m},\\
&+\frac{\mu p}{4c(n_1)}\bigl[\vec m(h_1-2h_1h_2+2\vec{\ell}\cdot\vec{m}) -\vec{\ell}h_2 \bigr]\cdot\n,
\end{split}
\eeno which gives rise to  \beq \label{2.23}
\begin{split}
\p_X[\vec{m}\cdot\n]=\f{p}{8c^3(n_1)}\bigl[(c^2(n_1)-\al)(h_1+h_2-2h_1h_2)-2(3c^2(n_1)-\al)\vec{\ell}\cdot\vec{m}\bigr](|\n|^2-1)\\
\qquad-\f{c'(n_1)}{4c^2(n_1)}pm_1[\vec{\ell}\cdot\n-\vec{m}\cdot\n]+\frac{\mu p}{4c(n_1)}(h_1-2h_1h_2+2\vec{\ell}\cdot\vec{m})(\vec m\cdot\n)  -\frac{\mu p h_2}{4c(n_1)} (\vec{\ell}\cdot\n).
\end{split}
\eeq While it is easy to observe that \beq \label{2.24}
\p_Y[|\n|^2-1]=\f{q}{c(n_1)}\vec{m}\cdot\n. \eeq Finally to control
the evolution of $|\vec{m}|^2+h_2^2-h_2,$ we get by applying
\eqref{2.19} that \beno
\begin{split}
\p_X\bigl[|\vec{m}|^2+h_2^2-h_2\bigr]=&\f{p}{4c^3(n_1)}\bigl[(c^2(n_1)-\al)(h_1+h_2-2h_1h_2)-2(3c^2(n_1)-\al)\vec{\ell}\cdot\vec{m}\bigr]\vec{m}\cdot\n\\
&+\f{\al-\ga}{4c^3(n_1)}p\bigl[(h_1+h_2-2h_1h_2)-2\vec{\ell}\cdot\vec{m}\bigr]m_1n_1\\
&+\f{c'(n_1)}{2c^2(n_1)}pm_1\bigl[(|\vec{m}|^2-\vec{\ell}\cdot\vec{m})+(h_2-\f12)(h_2-h_1)\bigr]\\
&+\frac{\mu p}{2c(n_1)}\vec m \cdot\bigl[2\vec m(h_1-h_1h_2+\vec{\ell}\cdot\vec{m})-\vec{\ell} h_2 -\vec{m}h_1 \bigr]\\
&+\frac{\mu p}{2c(n_1)}(2 h_2^2-h_2)(h_1-h_1h_2+\vec{\ell}\cdot\vec{m})
\end{split}
\eeno which leads to \beq\label{2.25}
\begin{split}
\p_X\bigl[|\vec{m}|^2+h_2^2-h_2\bigr]=
&\f{p}{4c^3(n_1)}\bigl[(c^2(n_1)-\al)(h_1+h_2-2h_1h_2)-2(3c^2(n_1)-\al)\vec{\ell}\cdot\vec{m}\bigr]\vec{m}\cdot\n\\
&+\bigl(\f{c'(n_1)}{2c^3(n_1)}pm_1+ \frac{\mu p}{2 c(n_1)}[h_1-2h_1h_2+2\vec{\ell}\cdot\vec{m}]\bigr)\bigl[|\vec{m}|^2+h_2^2-h_2\bigr]
\end{split}
\eeq Summing up \eqref{2.21} to \eqref{2.25} gives rise to
\eqref{2.20}. This completes the proof of the proposition.
\end{proof}

%
\section{Solutions in the energy coordinates}
\setcounter{equation}{0}
In this section, we prove the existence of the solution for
(\ref{2.19}) with boundary data converted from
(\ref{ID}). To avoid the confusion,
the reader should aware that in this section we solve variables $u,\ p,\ q,\ \vec\ell,\ \vec{m},\ h_1$ and $h_2$ 
by system (\ref{2.19}) instead of (\ref{1.1}) and we do not use the definitions (\ref{2.6b}) and (\ref{2.11})
except when we assign the boundary data.

The initial line $t=0$ in the $(t,x)$-plane is transformed to a
parametric curve 
\beq\label{1.33} \gamma: \quad Y = \varphi (X) \eeq
in the $(X, Y)$ plane,  where $Y = \varphi( X )$ if and only if
there is an $x$ such that \beq\label{1.34} \left\{
\begin{array}{rcl}
X &=& \int_0^x [1+|\vec{R}|^2(0,y)]\,dy, \\[3mm]
 Y &=& \int_x^0[1+|\vec{S}|^2(0,y)]\,dy.
 \end{array}\right.
\eeq The curve is non-characteristic. We introduce 
\beq\label{E.0}
\mathcal{E}_0
~\eqdefa~ \frac14\int \big[ |\vec{R}|^2(0,y)+|\vec{S}|^2(0,y)\big]\,dy<\infty. 
\eeq 
It equals to the number in (\ref{1.5E}) with $t=0$. The
two functions $X= X(x), Y = Y(x)$ from (\ref{1.34}) are well-defined
and absolutely continuous, provided that \eqref{ID} is satisfied. 
So
$\varphi (X)$ is continuous and strictly decreasing on $X$ since $X(x)$ is strictly
increasing while $Y(x)$ is strictly decreasing.  From
(\ref{E.0}) it follows \beq \big|X+\varphi (X)\big|  \leq 4 {\mathcal
E}_0\,.\label{xp} \eeq As $(t,x)$ ranges over the domain
$\mathbb R^+\,\times{\mathbb R}$, the corresponding variables $(X,
Y)$ range over the set \beq \Omega^+ := \big\{ (X,Y)\,; ~~Y\geq
\varphi(X)\big\}\,.\label{2.27} \eeq Along the curve
$$
\gamma :=   \big\{ (X,Y)\,;~~Y=\varphi (X)\big\}\subset {\mathbb
R}^2
$$
parametrized by $x\mapsto \big(X(x), \,Y(x)\big)$, we can thus
assign the boundary data $(\bar{\vec{\ell}}, \bar{\vec{m}}, \bar h_1, 
\bar h_2, \bar p, \bar q, \bar\n)$ $\in
L^\infty$ defined by their definition evaluated at the initial data
(\ref{ID}), i.e., \beq
\begin{split}
&{\bar \n}  =  \n_0(x)\,, \quad {\bar p} = 1\,,\quad {\bar q} = 1\,,\\
&\bar{\vec{\ell}}= \vec{R}(0,x)\bar h_1,\quad \bar{\vec{m}}=\vec{S}(0, x)\bar h_2,\\
&\bar h_1 = \frac1{1+|\vec{R}|^2(0, x)}, \qquad\mbox{and}\\
&\bar h_2 = \frac1{1+|\vec{S}|^2(0, x)}.
\end{split} \label{2.28}\eeq where 
\[
 \vec{R}(0, x) =  \n_1(x)+c({n_1}_0(x))\n_0'(x),   \quad  \vec{S}(0, x) =  \n_1(x)-c({n_1}_0(x))\n_0'(x).
 \]

We consider solutions to the boundary value problem
(\ref{2.19})(\ref{2.28})(\ref{n1}).

\begin{thm} The problem
(\ref{2.19})(\ref{2.28})(\ref{ID})(\ref{n1}) has a unique global solution defined for
all $(X, Y)\in \Omega^+$.
\end{thm}

\noindent{\it Sketch of Proof}. 
Noticing that all equations in \eqref{2.19} have a locally
Lipschitz continuous right hand side, the construction of a local
solution as fixed point of a suitable integral transformation is
straightforward. 
Note the consistency condition $\p_X[\p_Y\n]=\p_Y[\p_X\n]$ proved in the previous section 
shows that we can use either the equation for $\p_Y\n$
or the one for $\p_X\n$ to find the solution.

To make sure that this solution is actually defined
on $\Omega^+$, one must establish {\it a priori}
bounds, showing that the solution remains bounded on bounded subset of $\Omega^+$.

By (\ref{2.20}), we have
\beq\label{1.1618} h_1(1-h_1)\geq0, \quad h_2(1-h_2)\geq0. \eeq 
Thus $h_1, h_2$ are bounded
between zero and one, and $|\vec{\ell}|$ and $|\vec{m}|$ are
both uniformly bounded by (\ref{2.20}). 

By the $p$ and $q$ equations in (\ref{2.19}) and (\ref{2.20}), we have 
\beq
p_Y+q_X=-\frac{\mu p q h_1 h_2}{2c(n_1)}\, \, |\frac{\vec\ell}{h_1}+\frac{\vec m}{h_2}|^2\leq0
\eeq
which implies that 
\beno \int_{\varphi^{-1}(Y)}^X p(X', Y)dX' +
\int_{\varphi(X)}^Yq(X, Y')dY' \leq X - \varphi^{-1}(Y)+ Y - \varphi(X)
\eeno where $\varphi^{-1}$ denotes the inverse of $\varphi$,
following an integration over the characteristic triangle with
vertex $(X, Y)$. Thus, by the energy assumption (\ref{E.0}), we find
\beq\label{1pq} \int_{\varphi^{-1}(Y)}^X p(X', Y)dX' +
\int_{\varphi(X)}^Yq(X, Y')\,dY' \leq 2(|X|+|Y|+4{\mathcal E}_0). \eeq
Integrating the $p$ equation in \eqref{2.19} vertically and use the bound on $q$
from (\ref{1pq}), we find \beq\label{1p}
\begin{array}{rcl}
p(X, Y) &= &\exp\left\{\int_{\varphi(X)}^Y \frac{q(X, Y')}{2c(n_1)}\bigl[-\frac{c'(n_1)}
{2c(n_1)}\big( \ell_1-m_1\big)-\mu (h_2-h_1h_2+\vec{\ell}\cdot \vec{m})\bigr]\,dY'\right\}
\\[3mm]
&\leq & \exp\left\{C_0\int_{\varphi(X)}^Yq(X, Y')\,dY'\right\}\\[3mm]
&\leq& \exp\{2C_0(|X|+|Y|+4{\mathcal E}_0)\}.
\end{array}
\eeq Here $C_0$ represents a finite number. Similarly, we have
\beq\label{1q} q(X, Y) \leq \exp\{2C_0(|X|+|Y|+4{\mathcal E}_0)\}. \eeq

Relying on the local bounds (\ref{1p})(\ref{1q}), the local solution
to (\ref{2.19})(\ref{2.28})(\ref{ID})(\ref{n1}) 
can be extended to $\Omega^+$
One
may consult paper \cite{BZ} for details. This completes the sketch
of the proof.



\begin{col}\label{col1}
If the initial data $(\n_0, \n_1)$ are smooth,
the solution $U:=(\n,p,q,\vec{\ell}, \vec{m}, h_1, h_2)$
of (\ref{2.19})(\ref{2.28})(\ref{ID})(\ref{n1}) is a smooth function of the variables
$(X,Y)$. Moreover, assume that a sequence of smooth functions
$(\n_0^i, \n_1^i)_{i\geq 1}$ satisfies \beno \n_0^i\to \n_0\,,\qquad (\n_0^i)_x\to (\n_0)_x\,,\qquad \n_1^i\to \n_1 \eeno uniformly on compact subsets
of ${\mathbb R}$. Then one has the convergence of the corresponding
solutions:
$$(\n^i,p^i, q^i,{\vec{\ell}}^i, {\vec{m}}^i,h_1^i,h_2^i)\to U$$
uniformly on bounded subsets of $\Omega^+$.
\end{col}

\section{Inverse transformation}
\setcounter{equation}{0}

By expressing the solution $\mathbf n(X,Y)$ in terms of the original
variables $(t,x)$, we shall recover a solution of the Cauchy problem
\eqref{1.1}$\sim$\eqref{n1}. This will prove Theorem
\ref{1.1thm}, except the estimate $ {\mathcal{E}}(t)\leq{\mathcal{E}}_0$ 
which will be proved in the next section.

Using \eqref{2.6a} by letting $f=t$ or $x$, we have equations:
	\beq\label{3.1} 
 		t_X = \frac{p h_1}{2c}, \quad
		t_Y = \frac{q h_2}{2c}, \quad 
		x_X =\frac{ p h_1}{2}, \quad 
		x_Y = \frac{-q h_2}{2}. 
	\eeq
In fact, we only need one of the two equations of $t_X$ and $t_Y$, 
which are consistent since $t_{XY} =
t_{YX}$. The same is true for $x$.
We integrate \eqref{3.1} with data $t=0, x=x$ on $\gamma$ to find
$t = t(X, Y), x = x(X, Y)$, which exist for all $(X, Y)$ in
$\Omega^+$.

We need the inverse functions $X = X(t, x), Y = Y(t,x)$. The inverse
functions do not exist as a one-to-one correspondence between $(t,
x)$ in ${\mathbb R}^+\times\mathbb R$ and $(X, Y)$ in $\Omega^+ $. There may be
a nontrivial set of points in $\Omega^+ $ that maps to a
single point $(t, x)$. To investigate it, we find the partial
derivatives of the inverse mapping, valid at points where $h_1\ne 0,
h_2\ne 0$, 
	\beq\label{inverse} 
		X_t = \frac{c}{ph_1},\quad 
		Y_t = \frac{c}{qh_2},\quad 
		X_x = \frac1{ph_1},\quad 
		Y_x = -\frac1{qh_2}.
	\eeq 
Thus (\ref{2.6}) holds and so does (\ref{2.6a}) for our
solution.

We first examine the regularity of the solution constructed in the previous Section. Since
the initial data $({\mathbf n}_0)_x, {\mathbf n}_1$ etc. are only assumed to be in
$L^2$, the functions $U$ may well be discontinuous. More precisely,
on bounded subsets of the $\Omega^+$, the solutions satisfy the
following:
\begin{itemize}
\item The functions $\ell_1, \ell_2,,  \ell_3, h_1, p$ are Lipschitz continuous w.r.t.~$Y$,
measurable w.r.t.~$X$.
\item The functions $m_1, m_2, m_3, h_2, q$ are Lipschitz continuous w.r.t.~$X$,
measurable w.r.t.~$Y$.
\item The vector field ${\mathbf n}$ is Lipschitz continuous w.r.t.~both $X$ and $Y$.
\end{itemize}

In order to define $\mathbf n$ as  a vector field of the original variables
$t,x$, we should formally invert the map $(X,Y)\mapsto (t,x)$ and
write ${\mathbf n}(t,x)={\mathbf n}\big( X(t,x)\,,Y(t,x)\big)$. The fact that
the above map may  not be one-to-one does not cause any real
difficulty. Indeed, given $(t^*,x^*)$, we can choose an arbitrary
point $(X^*,Y^*)$ such that $t(X^*,Y^*)=t^*$, $x(X^*,Y^*)=x^*$, and
define ${\mathbf n}(t^*,x^*)={\mathbf n}(X^*,Y^*)$. To prove that the values
of $\mathbf n$ do not depend on the choice of $(X^*,Y^*)$, we proceed
as follows. Assume that there are two distinct points such that
$t(X_1, Y_1)=t(X_2,Y_2)=t^*$, $~x(X_1,Y_1)=x(X_2,Y_2)=x^*$. We
consider two cases:

\noindent Case 1: $X_1\leq X_2$, $Y_1\leq Y_2$. Consider the set
$$
\Gamma_{x^*} := \Big\{ (X,Y)\,;~~x(X,Y)\leq x^*\Big\}
$$
and call $\partial \Gamma_{x^*}$ its boundary. By (\ref{3.1}), $x$
is increasing with $X$ and decreasing with $Y$. Hence, this boundary
can be represented as the graph of a Lipschitz continuous function:
$X-Y=\phi(X+Y)$. We now construct the Lipschitz continuous curve
$\gamma$ consisting of
\begin{itemize}
\item a horizontal segment joining $(X_1,Y_1)$ with a point $A=(X_A,Y_A)$
on $\partial \Gamma_{x^*}$, with $Y_A=Y_1$,

\item a portion of the boundary $\partial \Gamma_{x^*}$,

\item a vertical segment joining $(X_2,Y_2)$ to a point $B=(X_B,Y_B)$
on $\partial \Gamma_{x^*}$, with $X_B=X_2$.
\end{itemize}
\noindent Observe that the map $(X,Y)\mapsto (t,x)$ is constant
along $\gamma$. By (\ref{3.1}) this implies $h_1=0$ on the
horizontal segment, $h_2=0$ on the vertical segment, and $h_1=h_2=0$
on the portion of the boundary $\partial \Gamma_{x^*}$. When either
$h_1=0$ or $h_2=0$ or both, we have from the conserved quantities
\eqref{2.21} that $\vec{\ell}=0$ or $\vec{m} =0$ or
both, correspondingly. Upon examining the derivatives of $\mathbf n$ in \eqref{2.19},
we have the same pattern of vanishing property. Thus, along the path
from $A$ to $B$, the values of the components of $\mathbf n$ remain constant, proving our claim.

\noindent Case 2: $X_1\leq X_2$, $Y_1\geq Y_2$.   In this case, we
consider the set
$$
\Gamma_{t^*} := \Big\{ (X,Y)\,;~~t(X,Y)\leq t^*\Big\}\,,
$$
and construct a  curve $\gamma$ connecting $(X_1,Y_1)$ with
$(X_2,Y_2)$ similarly as in case 1.  Details are entirely similar to
Case 1.

We now prove that the function ${\mathbf n}(t,x)=\mathbf n\big(X(t, x), Y(t,
x)\big)$ thus obtained are H\"older continuous on bounded sets.
Toward this goal, consider any characteristic curve, say  $t\mapsto
x^+(t)$, with $d x^+/dt =c(n_1)$.   By construction, this is
parametrized by the function $X\mapsto \big(t(X,\ov Y),\, x(X,\ov
Y)\big)$, for some fixed $\ov Y$. Using the chain rule and the
inverse mapping formulas (\ref{inverse}),
we obtain 
	\beno 
		{\mathbf n}_t+c{\mathbf n}_x = {\mathbf n}_X(X_t+cX_x)+{\mathbf n}_Y(Y_t+cY_x) = 2c X_x {\mathbf n}_X.
	\eeno 
Thus we have 
	\beq
		\begin{array}{rcl}
			\int_0^\tau \big| {\mathbf n}_t+c {\mathbf n}_x\big|^2\,dt 
			&= &\int_{X_0}^{X_\tau}(2cX_x|{\mathbf n}_X|)^2\,(2X_t)^{-1} dX\\[4mm]
			&=& \int_{X_0}^{X_\tau} (2c\frac1{ph_1}\frac{p |\vec{\ell}|}{2c})^2\frac{ph_1}{2c}\,dX \\[4mm]
			& = & \int_{X_0}^{X_\tau}{\frac{p}{2c}}\frac{ |\vec{\ell}|^2}{h_1}\,dX ~\leq ~\int_{X_0}^{X_\tau}{\frac{p}{2c}}\,dX~\leq ~C_\tau\,,
		\end{array}\label{4.4BZ}
	\eeq 
for some constant $C_\tau$ depending only on $\tau$. Notice we
have used  $ |\vec{\ell}|^2\leq h_1$, which follows from \eqref{2.20}.
Similarly, integrating along any backward characteristics $t\mapsto
x^-(t)$ we obtain 
\beq 
\int_0^\tau \big| {\mathbf n}_t-c {\mathbf n}_x\big|^2\,dt\leq
C_\tau. \label{4.5BZ} 
\eeq 
Since the speed of characteristics is
$\pm c(n_1)$, and $c(n_1)$ is uniformly positive and bounded, the bounds
(\ref{4.4BZ})-(\ref{4.5BZ}) imply that the function ${\mathbf n}={\mathbf n}(t,x)$ 
is H\"older continuous with exponent $1/2$.  In turn, this implies that
all characteristic curves are $\C^1$ with H\"older continuous
derivative. In addition, from (\ref{4.4BZ})-(\ref{4.5BZ})
it follows that $\vec{R}, \vec{S}$ at \eqref{2.1}
are square integrable on bounded subsets of the $t$-$x$ plane.
However, we should check the consistency that $\vec{R}, \vec{S}$ at \eqref{2.1} are indeed the same as recovered from
(\ref{2.11}). Let us check only one of them, $R = \vec{\ell}/h_1$. We
find
$$
{\mathbf n}_t+c{\mathbf n}_x = 2cX_x{\mathbf n}_X
=2c\frac1{ph_1}\frac{p\vec{\ell}}{2c}=\frac{\vec{\ell}}{h_1}.
$$

Finally, we prove that ${\mathbf n}$ satisfies the
equations of system \eqref{1.1} in distributional sense, according
to (iii) of Definition 1.1. We note that 
\beq
\begin{array}{rcl}
&& \dint 2\big[\phi_t {\mathbf n}_t - \phi_x c^2 {\mathbf n}_x\big] +2\mu \phi_t  \n\,dxdt \\[4mm]
&&=\dint \phi_t\big[({\mathbf n}_t+c{\mathbf n}_x)+({\mathbf n}_t-c{\mathbf n}_x)\big]\\
    &&\qquad -c\phi_x\big[({\mathbf n}_t+c{\mathbf n}_x)-({\mathbf n}_t-c{\mathbf n}_x)\big]+2\mu \phi_t  \n\,dxdt\\[4mm]
&&=\dint \big[\phi_t -c\phi_x\big]\,({\mathbf n}_t+c{\mathbf n}_x)\\
&&\quad+\dint \big[\phi_t
    +c\phi_x\big]\,({\mathbf n}_t-c{\mathbf n}_x)+2\mu \phi_t  \n\,dxdt\\[4mm]
&&=\dint \big[\phi_t -c\phi_x\big]\,\vec{R}\,dxdt+\dint \big[\phi_t
    +c\phi_x\big]\,\vec{S}\,dxdt+\dint 2\mu \phi_t  \n\,dxdt\,.
\end{array}\label{4.6BZ}
\eeq 
By \eqref{2.6a}, this is equal to 
\beq \dint
\big[-2cY_x \phi_Y \,\vec{R}~+~2cX_x\phi_X \,\vec{S}~+~2\mu c(X_x \phi_X-Y_x \phi_Y )\n\big]\,dxdt\,.
\label{4.7BZ} \eeq Using the Jacobian 
	\beq\label{J} 
		\frac{\partial(x, t)}{\partial(X, Y)} = \frac{pqh_1h_2}{2c} 
	\eeq
derived by
\eqref{3.1},
and the inverse
(\ref{inverse}), it is equal to 
\beq 
\begin{split} &\dint
\Big[\frac{2c}{qh_2}\vec{R}\phi_Y + \frac{2c}{ph_1}\vec{S}\phi_X + \frac{2\mu c}{p h_1} \phi_X + \frac{2\mu c}{q h_2} \phi_Y) \Big]\frac{pqh_1h_2}{2c}\,dXdY\\
=&\dint
\big[ph_1\vec{R}\phi_Y+qh_2\vec{S}\phi_X +\mu q h_2 \n \phi_X+\mu p h_1 \n \phi_Y
\big]\,dXdY
\\
=&\dint
\big[p\vec{\ell}\phi_Y+q\vec{m}\phi_X +\mu q h_2 \n \phi_X+\mu p h_1 \n \phi_Y
\big]\,dXdY
\\
=&\dint
\big[-(p\vec{\ell})_Y-(q\vec{m})_X -\mu(q h_2 \n)_X - \mu (p h_1 \n)_Y
\big]\phi\,dXdY.
\end{split} \label{formulation} \eeq

Thanks to \eqref{2.1}, \eqref{2.11}  and \eqref{2.19}, we have that the first component of the integrand equals to 
\beno \begin{split}
&\phi\bigl[-(p\ell_1)_Y-(q m_1)_X -\mu(q h_2 \n)_X - \mu (p h_1 \n)_Y\bigr]\\
=&-\phi\f{pq}{4c^3}\bigl[(c^2-\ga)(h_1+h_2-2h_1h_2)-2(3c^2-\ga)\vec{\ell}\cdot\vec{m}\bigr]n_1\\
=&-\phi\f{p q h_1 h_2}{2c}\bigl[\frac{c^2-\ga}{2c^2}(\frac{1}{h_2}+\frac{1}{h_1}-2)-\frac{3c^2-\ga}{c^2}\frac{\vec{\ell}\cdot\vec{m}}{h_1 h_2}\bigr]n_1
\\
=&-\phi\f{p q h_1 h_2}{2c}\bigl[\frac{c^2-\ga}{2c^2}(|\vec{R}|^2+|\vec{S}|^2)-\frac{3c^2-\ga}{c^2}\vec{R}\cdot\vec{S}\bigr]n_1
\\
=&-\phi\f{p q h_1 h_2}{2c}\bigl[\frac{c^2-\ga}{c^2}(|\n_t|^2+c^2|\n_x|^2)-\frac{3c^2-\ga}{c^2}(|\n_t|^2-c^2|\n_x|^2)\bigr]n_1.
\\=&-2\phi\f{p q h_1 h_2}{2c}\bigl(-|\n_t|^2+(2c^2-\ga)|\n_x|^2\bigr)n_1,
\end{split} \eeno  
which implies that the first equation in \eqref{1.1} holds
in integral form. Similarly, ${\mathbf n}$ satisfies the
second and third equations of system \eqref{1.1} in distributional sense.


\section{Upper bound on energy} 
\setcounter{equation}{0}

We convert the energy equation \eqref{1.4a2} formally to the
$(X,Y)$-plane to look for the upper bound of energy.

Recall that the energy equation  \eqref{1.4a2} can be written as   \eqref{2.3} in terms of 
$|\vec{R}|$ and $|\vec{S}|$.
By the variables (\ref{2.11}), we rewrite the equation \eqref{2.3} as 
\beq\label{e_ineq}
\left(\frac1{4h_1}+\frac1{4h_2}-\frac12\right)_t -
\left[\frac{c}{4}(\frac1{h_1} - \frac1{h_2})\right]_x \leq0. 
\eeq 
We can write the $1$-form 
\beq\label{5.3closed}
\left(\frac1{4h_1}+\frac1{4h_2}-\frac12\right)dx
+\left[\frac{c}{4}(\frac1{h_1} - \frac1{h_2})\right]dt \eeq
 as 
\beq\label{energyXY} \frac{p(1-h_1)}{4}dX -
\frac{q(1-h_2)}{4}dY, 
\eeq 
by the formula 
\beq\label{dxdt}
\begin{array}{rcccl}
dt & = & t_XdX+ t_YdY & = & \frac{ph_1}{2c}dX + \frac{qh_2}{2c}dY\\[2mm]
dx & = & x_XdX+ x_YdY & = & \frac{ph_1}{2}dX - \frac{qh_2}{2}dY.
\end{array}
\eeq 
By \eqref{2.19} then by \eqref{2.20}, we have 
\ben
\frac 1 4\big[ \bigl(p(1-h_1)\bigl)_Y +
\bigl(q(1-h_2)\bigl)_X\big]&=&-\frac{pq\mu}{2c}(h_1+h_2-2h_1h_2+2 \vec\ell\cdot\vec m)\nn\\
&=&-\frac{pq\mu}{2c}(|\vec\ell+\vec m|^2+(h_1-h_2)^2)\label{5.3b}\\
&\leq& 0.\nn 
\een
\begin{figure}[htb]
\centering
\includegraphics[width=0.45\textwidth]{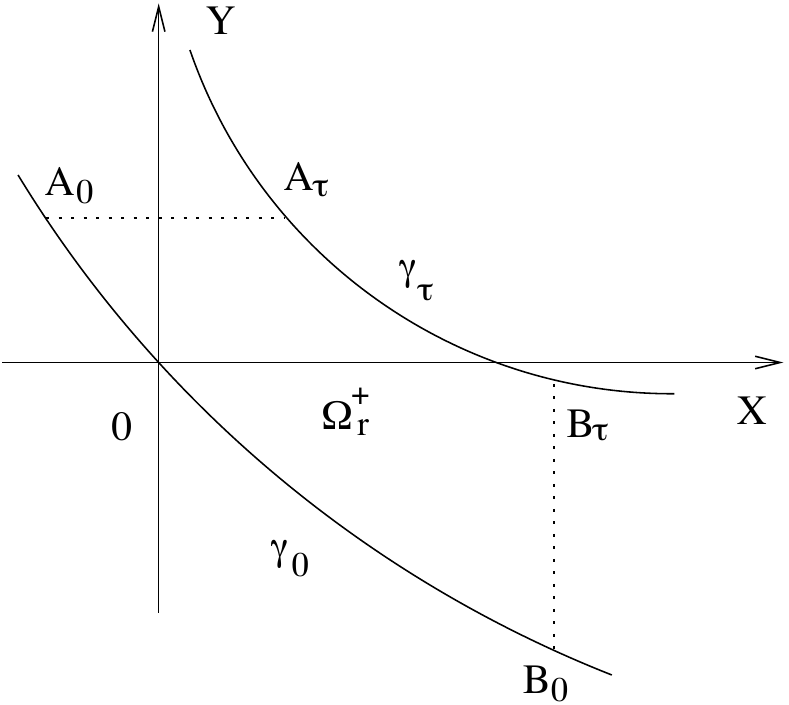}
\caption{Energy conservation } \label{f1}
\end{figure}

We use (\ref{energyXY}) to prove that the 
\beq\label{1.7BZ}
{\mathcal{E}}(t)\leq{\mathcal{E}}_0 \quad \mbox{for\ any}\quad t>0. 
\eeq
Fix $\tau>
0,$ and $ r>> 1$. Define the set 
\beq\label{Omega} \Omega^+_r := \Big\{
(X,Y)\,;~~0\leq t(X,Y)\leq \tau\,,\quad X\leq r\,, ~~Y\leq r\Big\}.
\eeq See Figure \ref{f1}, where segment $A_\tau A_0$ is where $Y=r$ while
segment $B_\tau B_0$ is where $X=r$. By construction, the map $(X,Y)\mapsto
(t,x)$ will act as follows:
$$
A_\tau\mapsto (\tau, a)\,,\qquad B_\tau\mapsto (\tau, b)\,,\qquad
B_0\mapsto(0,c)\,, \qquad A_0\mapsto (0,d)\,,
$$
for some $a<b$ and $d<c$. 
By the divergence theorem and \eqref{5.3closed}, the integral of the form
(\ref{energyXY}) along the curve $A_\tau\rightarrow B_\tau\rightarrow B_0\rightarrow  A_0 $ is less than or equal to zero.
Integrating the 1-form (\ref{energyXY})
along the boundary of $\Omega_r^+$ we obtain \beq\label{5.7BZ}
\begin{array}{lll}
 &&\int_{A_\tau B_\tau} \frac{p(1-h_1)}{4}\,dX  - \frac{q(1-h_2)}{4}\,dY\, \\[4mm]
&\leq&\int_{A_0 B_0} \frac{p(1-h_1)}{4}\,dX -\frac{q(1-h_2)}{4}\,dY\,
-\int_{A_0 A_\tau} \frac{p(1-h_1)}{4}\,dX - \int_{B_0 B_\tau}\frac{q(1-h_2)}{4}\,dY\\[4mm]
&\leq& \int_{A_0 B_0} \frac{p(1-h_1)}{4}\,dX- \frac{q(1-h_2)}{4}\,dY\, \\[4mm]
&=&\int_d^c \frac12\Big[ |{\n}_t|^2(0,x)+c^2\big({n_1}(0,x)\big)\,|\n_x|^2(0,x)\Big]\,dx
\,.
\end{array}
\eeq On the other hand, we use (\ref{dxdt}) to compute
\beq\label{5.8BZ}
\begin{split}
&\int_a^b \frac12\Big[  |{\n}_t|^2(\tau,x)+c^2\big({n_1}(\tau, x)\big)\,|\n_x|^2(\tau,x)\Big]\,dx \\[2mm]
&=\int_{A_\tau B_\tau\cap\{h_1\ne 0\}} \frac{p(1-h_1)}{4}dX -
\int_{A_\tau B_\tau\cap\{h_2\ne 0\}}\frac{q(1-h_2)}{4}dY \leq {\mathcal{E}}_0.
\end{split}
\eeq 
Letting $r\to +\infty$ in (\ref{Omega}),
one has $a\to -\infty$, $b\to +\infty$. Therefore (\ref{5.7BZ}) and
(\ref{5.8BZ}) together imply \eqref{1.7BZ}.

\section{Regularity of trajectories}  

\setcounter{equation}{0}

\subsection{Lipschitz continuity}
In this part, we first prove the Lipschitz continuity of the map $t\mapsto(n_1,n_2,n_3)(t,\cdot)$ in the $L^2$ distance, stated in (\ref{1.lip}).
 For any $h>0$, we have
\beno \n(t+h, x) - \n(t, x) = h\int_0^1 {\n}_t(t+\tau h, x)\,d\tau. \eeno
Thus 
\beq\label{5.14BZ} 
\| n_i(t+h, x) - n_i(t, x)\|_{L^2} \leq h\int_0^1
\| {n_i}_t(t+\tau h, \cdot)\|_{L^2}\,d\tau \leq h \sqrt{2{\mathcal E}_0}, \quad i=1\sim3.
\eeq 

\subsection{Continuity of derivatives}
We prove the continuity of functions $t\mapsto ({n_1}_t,
{n_2}_t,{n_3}_t)(t,\cdot)$ and $t\mapsto ({n_1}_x, {n_1}_x, {n_1}_x)(t,\cdot)$, as functions with
values in $L^p,\ 1\leq p<2$. (To keep tradition, we use the exponent $p$ here
at the expense of repeating one of our primary variables.) This will
complete the proof of Theorem \ref{1.1thm}.

We first consider the case where the initial data $(\n_0)_x$ and $\n_1$
are smooth with compact support. In this case,
the solution $\n =\n(X,Y)$ remains smooth in $\Omega^+$. Fix a time $\tau>0$. We claim that 
\beq\label{6.1BZ}
{\frac{d}{dt}} \n(t,\cdot)\bigg|_{t=\tau}= \n_t(\tau,\cdot) 
\eeq 
where
\beq\label{6.2BZ} \n_t (\tau,x)  :=  \n_X\,X_t+\n_Y\,Y_t
  =  \frac{p\,\vec{\ell}}{2c}\, \frac{c}{p h_1} +\frac{q \vec{m}}{2c}\,
\frac{c}{qh_2}=\frac{\vec\ell}{2h_1} +\frac{\vec m}{2h_2}\,. 
\eeq 
Notice
that (\ref{6.2BZ}) defines the values of $\n_t(\tau,\cdot)$ at almost
every point $x\in\R$. By the inequality (\ref{1.7BZ}), we obtain
\beq\label{6.3BZ} 
\int_{\mathbb R} \big |\n_t(\tau,x)\big|^2\,dx\leq 2\,
{\mathcal E}(\tau) \leq 2\,{\mathcal E}_0. 
\eeq

To prove (\ref{6.1BZ}), we consider the set \beq\label{5.11BZ}
\Gamma_\tau := \{(X, Y)\, |\, t(X, Y) \leq \tau \}, \eeq and let
$\gamma_\tau$ be its boundary. Let $\ve>0$ be given. There exist
finitely many disjoint intervals $[a_i,\,b_i]\subset \R$,
$i=1,\ldots,N$, with the following property. Call $\mathcal{A}_i, \mathcal{B}_i$ the
points on $\gamma_\tau$ such that $x(\mathcal{A}_i)=a_i$, $x(\mathcal{B}_i)=b_i$.  Then
one has \beq\label{6.4BZ} \min\big\{ h_1(P), h_2(P)\big\} <  2\ve
\eeq at every point $P$ on $\gamma_\tau$ contained in one of the
arcs $\mathcal{A}_i\mathcal{B}_i$, while \beq\label{6.5BZ} h_1(P)>\ve\,,\qquad\qquad
h_2(P)>\ve\,, \eeq for every point $P$ along $\gamma_\tau$, not
contained in any of the arcs $\mathcal{A}_i\mathcal{B}_i$. Call $J := \cup_{1\leq i\leq
N} [a_i, b_i]$, $J'=\R\setminus J$, and notice that, as a function
of the original variables, $\n=\n(t,x)$ is smooth in a neighborhood of
the set $\{\tau\}\times J'$. Using Minkowski's inequality and the
differentiability of $\n$ on $J'$, we can write, for $i=1\sim3$, 
\beq\label{6.6BZ}
\begin{array}{rl}
\lim_{h\to 0}& {\frac{1}{h}}\left(\int_{\R} \Big|
n_i(\tau+h,x)-n_i(\tau,x)-h\,{n_i}_t(\tau, x)\Big|^p
dx\right)^{1/p}\\[4mm]
&\leq \lim_{h\to 0} {\frac{1}{h}}\left(\int_J \Big|
n_i(\tau+h,x)-n_i(\tau,x)\Big|^p dx\right)^{1/p}\\
&\quad+\left(\int_J \big|{n_i}_t(\tau,x)\big|^p\,dx\right)^{1/p}\, .
\end{array}
\eeq We now provide an estimate on the measure of the ``bad" set
$J$: \beq\label{6.7BZ}
\begin{array}{rcl}
\meas(J)&=&\int_J dx=\sum_i\int_{\mathcal{A}_i\mathcal{B}_i} \frac{ph_1}{2}\,dX -\frac{qh_2}{2}\,dY\\[4mm]
&\leq & \frac{2\ve}{1-2\ve} \sum_i \int_{\mathcal{A}_i\mathcal{B}_i}
\frac{p(1-h_1)}{2}\,dX -\frac{q(1-h_2)}{2}\,dY\\[4mm]
&\leq& \frac{4\ve}{1-2\ve}  \int_{\gamma_\tau}
\frac{p(1-h_1)}{4}\,dX -\frac{q(1-h_2)}{4}\,dY \leq
\frac{4\ve}{1-2\ve}\,{\mathcal E}_0\,.
\end{array}
\eeq Notice that $dt=0$ on $\gamma_\tau$, so the two parts of the
integral are actually equal. Now choose $q=2/(2-p)$ so that ${\frac p
2}+{\frac 1 q}=1$. Using H\"older's inequality with conjugate
exponents $2/p$ and $q$, and recalling (\ref{5.14BZ}), we obtain for any $i=1\sim3$
$$
\begin{array}{rcl}
&&\int_J \Big| n_i(\tau+h,x)-n_i(\tau,x)\Big|^p dx \\
 &&
\leq   \meas(J)^{1/ q}\cdot \left(\int_J \Big| n_i(\tau+h,x)-n_i(\tau,x)\Big|^2 dx\right)^{p/2}\\[4mm]
&&\leq    \meas(J)^{1/ q}  \cdot
\Big(\big\|n_i(\tau+h,\cdot)-n_i(\tau,\cdot)\big\|^2_{L^2}
\Big)^{p/2}\\[4mm]
&&\leq \meas(J)^{1/ q} \cdot \Big(h^2 \big[ 2 {\mathcal E}_0\big]
\Big)^{p/2}.
\end{array}
$$
Therefore, \beq\label{6.8BZ} \begin{split} & \limsup_{h\to 0}
{\frac 1 h}\left(\int_J \Big| n_i(\tau+h,x)-n_i(\tau,x)\Big|^p
dx\right)^{1/p}\\
&\leq [\frac{4\ve}{1-2\ve}\,{\mathcal E}_0]^{1/pq}\cdot \big[ 2 {\mathcal
E}_0 \big]^{1/2}.\end{split} \eeq In a similar way we estimate
$$
\int_J\big|{n_i}_t(\tau,x)\big|^p\,dx\leq \big[\meas(J)\big]^{1/ q}\cdot
\left(\int_J \Big| {n_i}_t(\tau,x)\Big|^2 dx\right)^{p/2},
$$
\beq\label{6.9BZ}
\left(\int_J\big|{n_i}_t(\tau,x)\big|^p\,dx\right)^{1/p}\leq
\meas(J)^{1/ pq}\cdot \big[2 {\mathcal E}_0\big]^{p/2}. \eeq Since
$\ve>0$ is arbitrary, from (\ref{6.6BZ}), (\ref{6.8BZ}) and
(\ref{6.9BZ}) we conclude 
\beq\label{6.10Z} \lim_{h\to 0}~ {\frac 1
h}\left(\int_{\R} \Big| n_i(\tau+h,x)-n_i(\tau,x)-h\,{n_i}_t(\tau, x)\Big|^p
dx\right)^{1/p} =0. \eeq 
The proof of continuity of the map
$t\mapsto {n_i}_t$ is similar. Fix $\ve>0$. Consider the intervals
$[a_i,b_i]$ as before. Since $\n$ is smooth on a neighborhood of
$\{\tau\}\times J'$, it suffices to estimate
$$
\begin{array}{cl}
& \limsup_{h\to 0}\,\int\big|{n_i}_t(\tau+h,x)-{n_i}_t(\tau,x)\big|^p\,dx \\[4mm]
&\leq \limsup_{h\to 0}
\int_J \big|{n_i}_t(\tau+h,x)-{n_i}_t(\tau,x)\big|^p\,dx\\[4mm]
&\leq  \limsup_{h\to 0} \big[\meas(J)\big]^{1/ q}\cdot
\left(\int_J \Big| {n_i}_t(\tau+h,\,x)-{n_i}_t(\tau,x)\Big|^2 dx\right)^{p/2} \\[4mm]
&\leq  \limsup_{h\to 0} \big[\frac{4\ve}{1-2\ve} {\mathcal E}_0\big]^{1/
q}\cdot \Big(\big\|{n_i}_t(\tau+h,\cdot)\big\|_{L^2}
+\big\|{n_i}_t(\tau,\cdot)\big\|_{L^2}\Big)^p\\[4mm]
&\leq  \big[\frac{4\ve}{1-2\ve} {\mathcal E}_0]^{1/q}\, \big[
4{\mathcal E}_0\big]^p.
\end{array}
$$
Since $\ve>0$ is arbitrary, this proves continuity.

To extend the result to general initial data, such that $(\n_0)_x,\ 
{\n}_1=\n_t|_{t=0} \in L^2$, we use Corollary \ref{col1} and consider
a sequence of smooth initial data, with $({\n}^\nu_0)_x, {\n}^\nu_1  \in \C^\infty_c$, 
with ${\n}_0^\nu\to {\n}_0$ uniformly, $({\n}_0^\nu)_x \to
(\n_0)_x$ almost everywhere and in $L^2$, $\n_1^\nu 
\to \n_1$ almost everywhere and in $L^2$.

The continuity of the function $t\mapsto \n_x(t,\cdot)$
as maps with values in $L^p$, $1\leq p <2$, is proved in an
entirely similar way.
\bigskip

\noindent {\bf Acknowledgments.} Yuxi Zheng is partially supported by
NSF DMS 0908207.
%


\bigskip
\begin{thebibliography}{50}


\bibitem{AH} Giuseppe Al\`{i} and John Hunter, Orientation waves in a director field
with rotational inertia, {\it Kinet. Relat. Models}, {\bf 2} (2009),
 1-37.

\bibitem{[3]} H.~Berestycki, J.~M.~Coron and I.~Ekeland (eds.),
{\it Variational Methods}, Progress in Nonlinear Differential
Equations and Their Applications, Vol.~4, Birkh\"auser, Boston
(1990).


\bibitem{BZ} A.  Bressan and Yuxi Zheng,  Conservative solutions to a
nonlinear variational wave equation,
 {\it  Comm. Math. Phys.}, {\bf 266} (2006), 471--497.

\bibitem{CZZ12} Geng Chen, Ping Zhang and Yuxi Zheng, Energy Conservative Solutions to a Nonlinear Wave
System of Nematic Liquid Crystals,  {\it Comm. Pure Appl. Anal.}, (2012) in press.

\bibitem{[6]} D.~Christodoulou and A.~Tahvildar-Zadeh,  On the regularity of
spherically symmetric wave maps, {\it Comm. Pure Appl. Math.}, {\bf
46}(1993), 1041--1091.

\bibitem{[7]} J.~Coron, J.~Ghidaglia, and F.~H\'elein (eds.), {\it Nematics},
Kluwer Academic Publishers, 1991.

\bibitem{[10]} J.~L.~Ericksen and D.~Kinderlehrer (eds.),
{\it Theory and Application of Liquid Crystals}, IMA Volumes in
Mathematics and its Applications, Vol.~5, Springer-Verlag, New York
(1987).

\bibitem{ghz}  Robert~Glassey, John~Hunter and Yuxi~Zheng,
Singularities of a variational wave equation, {\it J. Differential Equations}, {\bf 129}(1996), 49--78.




\bibitem{[19]} R.~Hardt, D.~Kinderlehrer, and Fanghua Lin,
Existence and partial regularity of static liquid crystal
configurations. {\it Comm. Math. Phys.}, {\bf 105}(1986),
547--570.

\bibitem{[26]} D.~Kinderlehrer,  Recent developments in liquid crystal theory,
in {\it  Frontiers in pure and applied mathematics : a collection of
papers dedicated to Jacques-Louis Lions on the occasion of his
sixtieth birthday}, ed. R.~Dautray, Elsevier, New York, 151--178
(1991).

\bibitem{Mc} A.~McNabb, Comparsion Theorems for Differential Equations, 
{\it J. Math. Anal. Appl.}, {\bf 119}(1986), pp.~417--428.

\bibitem{[37]} R.~A.~Saxton, Dynamic instability of the liquid crystal
director, in {\it Contemporary Mathematics} Vol.~100: Current
Progress in Hyperbolic Systems, pp.~325--330, ed. W.~B.~Lindquist,
AMS, Providence, 1989.

\bibitem{[40]} J.~Shatah, Weak solutions and development of singularities
in the $SU(2)$ $\sigma$-model, {\it Comm. Pure Appl. Math.}, {\bf
41}(1988), 459--469.

\bibitem{[41]} J.~Shatah and A.~Tahvildar-Zadeh,  Regularity of harmonic
maps from Minkowski space into rotationally symmetric manifolds,
{\it Comm. Pure Appl. Math.}, {\bf 45}(1992), 947--971.


\bibitem{[49]} E.~Virga, {\it Variational Theories for Liquid Crystals},
Chapman \& Hall, New York (1994).


\bibitem{ZZ03} Ping Zhang and Yuxi Zheng,  Weak solutions to a nonlinear variational
wave equation, {\it Arch. Ration. Mech. Anal.}, {\bf 166} (2003),
303--319.

\bibitem{ZZ05a} Ping Zhang and Yuxi Zheng,
Weak solutions to a nonlinear variational wave equation with general
data, {\it  Ann. I. H. Poincar\'e}, {\bf 22} (2005), 207--226.


\bibitem{ZZ10} Ping Zhang and Yuxi Zheng, Conservative
solutions to a  system of variational wave equations of nematic
liquid crystals, {\it Arch. Ration. Mech. Anal.}, {\bf 195}
(2010), 701-727.

\bibitem{ZZ11} Ping Zhang and Yuxi Zheng, Energy Conservative Solutions to a One-Dimensional Full Variational Wave
System,  {\it Comm. Pure Appl. Math.}, {\bf 65:5} (2012), 683-726.

\end{thebibliography}
\end{document}